\documentclass[11pt]{article}
 \usepackage{myownarticle}

\newcommand{\cC}[0]{\mathcal{C}}
\newcommand{\cF}[0]{\mathcal{F}}
\newcommand{\cG}[0]{\mathcal{G}}
\newcommand{\cH}[0]{\mathcal{H}}
\newcommand{\cL}[0]{\mathcal{L}}
\newcommand{\cM}[0]{\mathcal{M}}
\newcommand{\cN}[0]{\mathcal{N}}
\newcommand{\cO}[0]{\mathcal{O}}
\newcommand{\cP}[0]{\mathcal{P}}
\newcommand{\cS}[0]{\mathcal{S}}
\newcommand{\cT}[0]{\mathcal{T}}
\newcommand{\cW}[0]{\mathcal{W}}

\newcommand{\cZ}[0]{\mathcal{Z}}

\newcommand{\bbE}[0]{\mathbb{E}}
\newcommand{\bbP}[0]{\mathbb{P}}
\newcommand{\bbR}[0]{\mathbb{R}}
\newcommand{\bbW}[0]{\mathbb{W}}

\begin{document}

\title{ Causal transport on path space}
\author{Rama CONT and Fang Rui LIM\\
Mathematical Institute, University of Oxford}
\maketitle
\begin{abstract}
We study   properties of causal couplings for probability measures on the space of continuous functions.
We first provide a 
characterization of bicausal couplings between weak solutions of stochastic differential equations.
We then provide a complete description of all such bicausal Monge couplings. In particular, we show that bicausal Monge couplings of $d$-dimensional Wiener measures are induced by stochastic integrals of rotation-valued integrands.
 As an application, we give necessary and sufficient conditions for bicausal couplings to be induced by Monge maps and show that such bicausal Monge transports are dense in the set of bicausal couplings between laws of SDEs with 
regular  coefficients.
\end{abstract} 
\newpage

\tableofcontents

\newpage

\section{Optimal transport and its causal counterpart}
\label{section:introduction}

The theory of optimal transport of probability measures  and the associated Wasserstein metric has been the focus of a large body of research \cite{ambrosio2021lectures,peyre2019computational,rachev2006,villani2008optimal} with rich ramifications across various branches of mathematics and applications. 
In the context of stochastic processes, where one deals with probability measures on {\it filtered} spaces, a natural extension of the concept of transport is to incorporate  compatibility with the flow of information, leading to the concept of {\it causal transport} \cite{causalopttptlassalle}, defined as a transport which makes use, at any time $t$, only the information available at $t$. As shown by Lassalle \cite{causalopttptlassalle}, when the underlying space is  the space of trajectories of a stochastic process, the concept  of causal transport  has interesting connections with stochastic analysis. 

A natural example of causal transport is the one induced by a {\it non-anticipative} map. Denote by $\cW^d := C\deln{\sbr{0,1}, \mathbb{R}^d}$ the space of continuous $\mathbb{R}^d$-valued functions on $\sbr{0,1}$, with its canonical filtration $\del{\mathcal{F}_t}_{t \in \sbr{0,1}}$. A map $T : \cW^d \rightarrow \cW^d$ may be naturally interpreted as a $\mathbb{R}^d-$valued  functional  $F:\sbr{0,1} \times {\cal W}^d \mapsto F \del{\omega,t} := T\del{\omega}\del{t} \in \bbR^d$. Such a functional is said to be causal or {non-anticipative}    \cite{MR2652181,fliess1981} if $F\del{\omega,t}$ depends on  the path $\omega$ only through its values on $[0,t]$:  \[
F (\omega,t) = F (\omega(\cdot \wedge t),t), \text{ where } s \wedge t := \min\cbr{s,t}.
\]  The causality of $F$ is equivalent to $T^{-1}\del{\cF_t} \subseteq \cF_t$ for all $t \in \sbr{0,1}$.
Any probability measure $\eta$ on $\cW^d,$ is thus transported in a causal manner, i.e. non-anticipative with respect the filtration, to  an image measure $\nu $ on on $\cW^d.$ The map $T$ is then an example of causal (Monge) coupling on $\cW^d \times \cW^d$ between $\eta $ and $\nu$.

An equivalent manner to state that $T : \cW^d \rightarrow \cW^d$  is causal is to require \begin{equation}
\label{eqn:nonanticipativeequivctd}
\forall  t \in \sbr{0,1},\quad  \forall  B \in \mathcal{F}_t, \; \omega \mapsto \delta_{T\del{\omega}}\del{B} \text{ is } \mathcal{F}_t\text{-measurable},
\end{equation} where $\delta_{T\del{\omega}}$ is the point mass at $T\del{\omega} \in \cW^d$.
Denoting by $X$ the canonical process on $\cW^d$, $\delta_{T\del{\omega}}$ 
 represents in fact the `conditional distribution' of $T(X)$ given $X=\omega$:
$$ \delta_{T\del{\omega}}=\cL^\eta \del{T(X) | X = \omega}$$
This suggests the following definition of causality for an {\it arbitrary} coupling: given a probability measure $\pi$ on $\cW^d \times \cW^d$ and writing $X$ and $Y$ as the first and second marginal processes, one  requires,  by analogy with \eqref{eqn:nonanticipativeequivctd}:
\begin{equation}
    \label{eqn:toycausalexplanation}
    \forall  t \in \sbr{0,1},\quad  \forall B \in \mathcal{F}_t, \; \omega \mapsto \mathcal{L}^\pi\del{Y| X = \omega}\del{B} \text{ is } \mathcal{F}_t\text{-measurable}.
\end{equation} There are some technical subtleties to take care of since the conditional distribution is only defined up to null sets, and there is a choice to be made between using the canonical filtration    or its right-continuous version. We will give a precise  definition below (Definition \ref{def:causaldefinition}).


Since the notion of causality is not symmetric, the corresponding optimal transport problem does not yield a (symmetric) notion of distance. A remedy is to consider \emph{bicausal} couplings, which satisfy \eqref{eqn:toycausalexplanation} and its symmetric counterpart \[  \forall  t \in \sbr{0,1},\quad \forall B \in \mathcal{F}_t, \; \omega \mapsto \mathcal{L}^\pi\del{X| Y = \omega}\del{B} \text{ is } \mathcal{F}_t\text{-measurable}. \] 
Causal    optimal transport has been studied in detail in the  discrete time case \cite{backhoff2017causal,backhoff2020all,eder2019compactness,bartl2021wasserstein,eckstein2024computational,backhoff2022estimating, adaptedwassersteindistancesbackhoff, xu2020cot,yfjiang2024dro}, where the  sequential character of processes allows to study the properties over a single time period and extend them by induction. 
However, the continuous time case reveals much more structure, with links to stochastic analysis and functional inequalities \cite{causalopttptlassalle,follmer2022}. 

As noted by Lassalle \cite{causalopttptlassalle},  the Ito map associated with the unique strong solution of a stochastic differential equation  (SDE) is a natural example of causal Monge transport from the Wiener measure to the law of the solution, while the concept of weak solution provides an example of causal Kantorovich transport on path space.
Optimal transport on the Wiener space has been investigated in \cite{feyel2004monge,fang2015}. Causal transport on Wiener space has been investigated in the context of filtration enlargements \cite{causalopttptacciaio} and its relation to Talagrand's inequality \cite{causalopttptlassalle,follmer2022}. Backhoff et al. \cite{adaptedwassersteindistancesbackhoff} show the relevance of  causal Wasserstein distances induced by such causal transports for superhedging problems. 

The present work is a detailed study  of the structure of (bi-)causal couplings on path space, focusing on the case where the measures involved have a `differential' structure i.e. are (weak) solutions of stochastic differential equations (SDEs).  In addition to confirming that many of the results obtained in the discrete-time setting remain true for (bi-)causal transports on path space, we show that, when the measures  are weak solutions of SDEs,   bicausal transports between these measures inherit this differential structure in many cases, and may be represented in terms of stochastic integrals or, more generally, semimartingales.

Related questions have been studied in \cite{backhoffveraguas2024adaptedwassersteindistancelaws,robinson2024bicausal}, where  (bi-)causal transport  problems for  solutions to \emph{scalar} SDEs are investigated, as well as in \cite{bion2019wasserstein} using a stochastic control approach. 

\subsection{Contributions}

We study   properties of causal couplings for probability measures on the space of continuous functions. We first prove a characterization of bicausal couplings between weak solutions of SDEs and then investigate the case where the bicausal couplings are induced by a Monge transport.  Denoting by the $\Pi_{c}\del{\eta,\nu}$ (resp. $\Pi_{bc}\del{\eta,\nu}$) the set of all causal (resp. bicausal) couplings, for any causal map $T : \mathcal{W}^d \rightarrow \mathcal{W}^d$ which transports $\eta$ to $\nu$, i.e. $T_{\#}\eta = \nu$, we have \[
\del{\text{Id},T}_{\#}\eta \in \Pi_c\del{\eta,\nu} \iff T \text{ is non-anticipative.}
\] We show in Lemma \ref{lemma:bicausalmongemapequivctd} that, when $\eta$ and $\nu$ are unique weak solutions of SDEs,  \[
\del{\text{Id},T}_{\#}\eta \in \Pi_{bc}\del{\eta,\nu} \iff \del{T_t}_{t \in \sbr{0,1}} \text{ is a } \del{\eta,\del{\mathcal{F}_t}_{t \in \sbr{0,1}}}\text{-semimartingale,}
\] where $T_t := T\del{\omega}\del{t}$. Using a  martingale representation result  of \"Ust\"unel \cite{ustunel2019martingale}, we are able to completely characterise  bicausal Monge transports between such measures. We show that such maps inherit the differential structure of $\eta,\nu$ in the sense that $T$ satisfies an SDE \eqref{eqn:bicausalmongemapeqn} (Theorem \ref{thm:bicausaltptmapsbetweenSDEsdegen}). In particular, we show bicausal Monge transports of $d$-dimensional Wiener measures are induced by stochastic integrals of adapted $\cO^d$-valued integrands.

Given an SDE with unique strong solution, the Ito map provides an example of causal transport from the Wiener measure $\bbW^d$ to the law $\nu$ of the solution on path space.
Theorem \ref{thm:bicausaltptmapsbetweenSDEsdegen} characterizes \emph{all} bicausal Monge maps  from $\bbW^d$ to $\nu$ as the composition of the Ito map with integration against an adapted ${\cal O}^d$-valued process.

Using the aforementioned characterization of bicausal Monge transports, we give simple conditions for their existence (Corollaries \ref{corollary:nobicausalmongemaps},\ref{corollary:strongsolnsdebicausalmongemap},\ref{corollary:existenceofbicausalmongemapssimpleeg}). We show that any bicausal coupling between Wiener measures has a differential representation i.e. may be represented by a stochastic differential equation (Proposition \ref{prop:bicausaltptplansinducedbybicausalmongemapswienermeasures}). Through this SDE, we can establish necessary and sufficient conditions for bicausal couplings between weak solutions of SDEs to be induced by a Monge map (Proposition \ref{prop:bicausaltptplansinducedbybicausalmongemapswienermeasures} and Corollary \ref{corollary:bicausaltptplansinducedbybicausalmongemapssdes}). 

Beiglb\"ock et al. \cite{beiglbock2018denseness}  showed that in discrete time, the set   $\mathcal{T}_{c}\del{\mathbb{P}, \mathbb{P}'}$ of causal Monge couplings is dense in the set of all  {causal} couplings.
We show that similar density properties hold for bicausal transports between laws of SDEs with strong, pathwise unique  solutions and continuous (possibly path-dependent) drift and invertible diffusion coefficients (Proposition \ref{prop:bicausalmongemapdenseinbicausaltptplans} and \ref{prop:bijectivebicausalmongemapdenseinbicausaltptplans})
. 
The proofs  make use of M. \'Emery's results on ``almost Brownian filtrations"\cite{emery2005certain}, whose potential role in the study of causal  transports was already noticed in \cite[Section $5.4$]{beiglbock2018denseness}. 

A consequence of our results is that, for many cost functions the bicausal Monge problem between weak solutions of SDEs and its Kantorovich relaxation lead to the same transport cost (Corollary \ref{corollary:equalityofbckpandbcmp}).

\paragraph{Outline}
 Section \ref{section:preliminaries}  establish notations and defines causal couplings and maps. Section \ref{section:causaltptsbetweensdes} contains the main results on the characterization of bicausal couplings and Monge maps between laws of solutions of SDEs, and some results regarding the existence of such transports. Section \ref{section:applications} considers some applications of these results.

\section{Causal couplings on path space}
\subsection{Notations}
\label{section:preliminaries}

Let $d \geq 1$. Given $v, w \in \bbR^d$, we denote the standard Euclidean inner product as $\del{v,w}_{\bbR^d}$ and norm $\norm{v}_{\bbR^d} := \sqrt{\del{v,v}_{\bbR^d}}$. The space of $d \times d$ matrices is written as $\bbR^{d \times d}$. $\text{Id}_d \in \bbR^{d \times d}$ is the identity matrix. Given $A \in \bbR^{d \times d}$, $A^*$ denotes the transpose of $A$, $A^{-1}$ denotes the inverse of $A$, $\Tr A$ denotes its trace, $\det A$ denotes its determinant, $\ker A$ denotes its kernel and $A^{\dag}$ denotes its Moore-Penrose pseudoinverse  \cite{golub2013matrix}. We denote $\cO^d$ - the set of orthogonal matrices,  $\cS^d$ the set of symmetric matrices; and $\cC^d\subset \cS^d$  the subset of correlation  matrices defined as \cite{emery2005certain}:  \[
C \in \mathcal{C}^d \text{ if and only if the $2d \times 2d$ matrix } \begin{pmatrix}
    \text{Id}_d & C \\
    C^* & \text{Id}_d 
\end{pmatrix} \text{ is positive semi-definite.}
\]

Denote $\mathcal{W}^d := C\deln{\sbr{0,1}, \mathbb{R}^d}$, the space of continuous $\mathbb{R}^d$-valued functions on $\sbr{0,1}$. Equipped with the topology $\tau_{\mathcal{W}^d}$ induced by the supremum norm, $\norm{\cdot}_\infty$, $\mathcal{W}^d$ is a Polish space with Borel $\sigma$-algebra $\mathcal{F}_1 := \sigma\del{\tau_{\cW^d}}$. The canonical filtration is denoted $\del{\mathcal{F}_t}_{t \in \sbr{0,1}}$ and given by $\mathcal{F}_t = \sigma\del{f_t} \subseteq \cF_1$, where $f_t : \cW^d \rightarrow \deln{\cW^d, \mathcal{F}_1}$, $f_t\del{\omega} := \omega\del{\cdot \wedge t}$ for $t \in \sbr{0,1}$ and set $\del{\mathcal{H}_t}_{t\in \sbr{0,1}}$ to be 
the right-continuous version of the canonical filtration, i.e. $\mathcal{H}_t = \displaystyle \cap_{ \epsilon > 0} \mathcal{F}_{(t + \epsilon) \wedge 1} \subseteq \cF_1$. The canonical process on $\mathcal{W}^d$ is defined as $X : \mathcal{W}^d \times \sbr{0,1} \rightarrow \mathbb{R}^d$, $X\del{\omega, t} = X_t\del{\omega} := \omega\del{t}$. When we write $X_{\cdot}$, we mean the identity map, i.e. $X_{\cdot} : \cW^d \rightarrow \cW^d$, $X_{\cdot}\del{\omega} = \omega$. On the product $\cW^d\times \cW^d$, without risk of confusion, $X$ will also denote the first marginal, $X(\omega, \omega') := \omega$ while $Y$ denotes the second marginal, $Y(\omega,\omega') := \omega'$. We refer to $\del{X_t}_{t \in \sbr{0,1}}$ (resp. $\del{Y_t}_{t \in \sbr{0,1}}$) as the first (resp. second) marginal process $X_t\del{\omega, \omega'} := \omega\del{t}$ (resp. $Y_t\del{\omega, \omega'} := \omega'\del{t}$) in this case.

Given measurable spaces $\del{Z_i, \mathcal{Z}_i}$, where $i = 1,2$ and $\mathcal{Z}_i$ is some $\sigma$-algebra, $\cZ_1 \otimes \cZ_2$ is the product $\sigma$-algebra on $Z_1 \times Z_2$. $\mathcal{P}\del{Z_i}$ denotes the space of probability measures on $Z_i$. Given $\eta \in \cP\del{Z_1}$, the $\eta$-completion of $\cZ_1$ is written as $\cZ_1^\eta$. Given a $\cZ_1/\cZ_2$-measurable function $f : Z_1 \rightarrow Z_2$, $f_{\#}\eta \in \cP\del{Z_2}$ is the pushforward of $\eta$ under $f$, i.e. $f_{\#}\eta\del{A} := \eta\del{f^{-1}\del{A}}$ for $A \in \mathcal{Z}_2$.  Furthermore, for $p \geq 1$, we write $L^p\del{\eta, \mathcal{Z}_1}$ to be the set of all $\mathcal{Z}_1$-measurable functions $f : Z_1 \rightarrow \bbR$ (identified up to $\eta$-a.s. equivalence) satisfying \[
\mathbb{E}^\eta \del{\envert{f}^p} := \int_{Z_1} \envert{f\del{\omega}}^p \eta\del{d\omega} < \infty.
\] A filtration on $(Z_1, \cZ_1)$ is a family of $\sigma$-algebras $(\cN_t)_{t \in \sbr{0,1}}$, such that $\cN_s \subseteq \cN_t \subseteq \cZ_1$ for all $0 \leq s \leq t \leq 1$. 
We call the quadruple $\del{Z_1, \mathcal{Z}_1, \del{\mathcal{N}_t}_{t \in \sbr{0,1}}, \eta}$ a stochastic basis if $\del{\mathcal{N}_t}_{t \in \sbr{0,1}}$ is a $\eta-$complete, right-continuous filtration. 
The space of all square-integrable $\deln{\eta,\del{\cN_t}_{t\in [0,1]}}$-martingales is denoted \[
\cM^2\del{\eta, \del{\cN_t}_{t\in[0,1]}} := \cbr{M=\del{M_t}_{t \in \sbr{0,1}} \text{ is an } \deln{\eta,\del{\cN_t}_{t \in \sbr{0,1}}}\text{-martingale},\quad \sup_{t \in \sbr{0,1}} \bbE^\eta \envert{M_t}^2 < \infty }.
\] Finally, if $\eta_i \in \cP\del{Z_i}$ for $i = 1,2$, $\eta_1 \otimes \eta_2$ denotes the product coupling $\eta_1\otimes \eta_2\del{A \times B} := \eta_1\del{A}\eta_2\del{B}$ for $A \in \cZ_1, B\in \cZ_2$.


\subsection{Definitions}
Let $\eta, \nu \in \cP\deln{\mathcal{W}^d}$ be probability measures.
 $\pi \in  \mathcal{P}\deln{\cW^d \times \cW^d}$ is a coupling of $(\eta,\nu)$ if \[
{X}_{\#} \pi = \eta \text{ and } {Y}_{\#} \pi = \nu.\]
The set of all such couplings is denoted $\Pi(\eta,\nu)$.
Since $\cW^d \times \cW^d$ equipped with the product topology is a Polish space, for any $\pi \in \Pi\del{\eta,\nu}$ there exists an $\eta$-a.s. unique probability kernel $\Theta_\pi : \cW^d \times \mathcal{F}_1 \rightarrow \sbr{0,1}$, $\del{\omega, A} \mapsto \Theta^\omega_\pi\del{A}$ which disintegrates $\pi$ with respect to $\eta$ \cite{kallenberg1997foundations}: \[
\pi\del{A \times B} = \int_{\cW^d} \mathbbm{1}_A\del{\omega} \Theta^\omega_\pi\del{B} \eta\del{d\omega}.
\] We call $\Theta_\pi$ the disintegration of $\pi$ with respect to $\eta$. Note that $\Theta^\omega_\pi$ is a version of the conditional distribution of $Y$ given $X = \omega$, $\cL^\pi\del{Y | X = \omega}$. We recall the definitions of causal transports and causal couplings \cite{causalopttptacciaio,causalopttptlassalle}:
\begin{definition}[Causal couplings and causal Monge maps]
\label{def:causaldefinition}\ \\
\begin{enumerate}
    \item $\pi \in \Pi\del{\eta,\nu}$ is a causal coupling between $\eta$ and $\nu$ if\begin{equation}
        \label{eqn:causaldef}
        \text{ for all } t \in \sbr{0,1} \text{ and } B \in \mathcal{H}_t, \; \omega \mapsto \Theta_\pi^\omega\del{B} \text{ is } \mathcal{H}_t^\eta\text{-measurable}
        .
    \end{equation} We denote by $\Pi_{c}\del{\eta,\nu}$ the set of causal couplings between $\eta$ and $\nu$;
    \item A measurable map $T : \cW^d \rightarrow \cW^d$ is a causal Monge transport between $\eta$ and $\nu$ if the image of $\eta$ under $T$ is $\nu$ and the induced transport plan $\pi_T := \del{X_{\cdot},T}_{\#}\eta$ is a causal transport: $\pi_T\in \Pi_{c}\del{\eta,\nu}$. 
    \item We say that $\pi \in \Pi\del{\eta, \nu}$ is a bicausal coupling between $\eta$ and $\nu$ if $\pi \in \Pi\del{\eta, \nu}$ and $R_{\#}\pi \in \Pi\del{\nu, \eta}$ where $R$ is the exchange of coordinates: \begin{eqnarray*}
        R : \cW^d \times \cW^d &\rightarrow &\cW^d \times \cW^d\nonumber\\
        (\omega',\omega)&\mapsto &R\deln{\omega,\omega'} := \deln{\omega',\omega}
    \end{eqnarray*}  The set of all bicausal couplings between $\eta$ and $\nu$  is denoted $\Pi_{bc}\del{\eta,\nu}$. 
    \item $T : \cW^d \rightarrow \cW^d$ is said to be a bicausal Monge map between $\eta$ and $\nu$ if $\pi_T \in \Pi_{bc}\del{\eta, \nu}$.
    
    $\cT_{bc}\del{\eta,\nu}$ denotes the set of all $\pi \in \Pi\del{\eta,\nu}$ such that $\pi = \pi_T$ for some bicausal Monge map $T$. Abusing notation, we will write $T \in \cT_{bc}\del{\eta,\nu}$ if $\pi_T$ is a bicausal coupling.
    \item Finally, $\mathcal{T}_{bcb}\del{\eta,\nu}$ denotes the set of all $\pi \in \Pi\del{\eta,\nu}$ such that $\pi = \pi_T$ for some \emph{a.s. bijective}\footnote{Here, $T$ is said to be a.s. bijective if there exists $\hat{T} : \cW^d \rightarrow \cW^d$ such that $\eta$-a.s. $\hat{T} \circ T =X_\cdot$ and $\nu$-a.s. $T \circ \hat{T} = X_\cdot$.} bicausal Monge map $T$. 
\end{enumerate}

\begin{remark}
\label{remark:causalcpls}
\mbox{}
\begin{enumerate}
    \item Comparing Definition \ref{def:causaldefinition} and \eqref{eqn:toycausalexplanation}, we observe two differences:
    \begin{itemize}
        \item the first is that we ask for measurability with respect to the completed $\sigma$-algebra $\mathcal{H}_t^\eta$, instead of just $\mathcal{H}_t$. This is because the disintegration $\Theta_\pi$ is only $\eta$-a.s. unique; and,
        \item secondly, we have used the right-continuous filtration $\del{\cH_t}_{t \in \sbr{0,1}}$, instead of $\del{\cF_t}_{t \in \sbr{0,1}}$. This is motivated by the fact $\del{\cH^\eta_t}_{t \in \sbr{0,1}}$ is a complete, right-continuous filtration and many familiar results from stochastic analysis hold, e.g. martingales have c\`adl\`ag modifications.
    \end{itemize}
    \item Notice that to verify \eqref{eqn:causaldef}, it is sufficient to check measurability of $\Theta_\pi\del{B}$ for all $B \in \cF_t$. This observation implies that:
    \begin{itemize}
        \item if $\cH_t^\eta = \cF_t^\eta$ for all $t \in \sbr{0,1}$, then it does not matter if we replaced $\cH_t$ with $\cF_t$ in \eqref{eqn:causaldef}. Indeed, the set of causal couplings would not change. This was observed in \cite[Remark $2.4$]{backhoffveraguas2024adaptedwassersteindistancelaws} for the case when $\eta$ is the law of a strongly Markov process. We also mention \cite[Theorem $5.19$]{liptser1977statistics} which shows that $\cH_t^\eta = \cF_t^\eta$ when $\eta$ is the law of an SDE with Lipschitz drift and uniformly elliptic diffusion coefficient; and,
        \item $\Pi_c\del{\eta,\nu}$ and hence, $\Pi_{bc}\del{\eta,\nu},$ are compact for the topology of weak convergence by \cite[Theorem $1$, Remark $5$]{causalopttptlassalle}.
    \end{itemize}
    \end{enumerate}
\end{remark}



\end{definition} 

\subsection{Relation with the H-hypothesis} Br\'emaud  and Yor \cite{bremaud1978changes} introduced the {\it H-hypothesis}  as a property of filtrations useful in filtering theory. Two filtrations $\del{\cM_t}_{t\in \sbr{0,1}}$ and $\del{\cN_t}_{t\in \sbr{0,1}}$ on a probability space $\del{\Omega, \cN, \bbP}$ such that $\cM_t \subseteq \cN_t \subseteq \cN$ are said to satisfy the H-hypothesis if  every square integrable $\deln{\bbP,\del{\cM_t}}$-martingale is also an $\deln{\bbP,\del{\cN_t}}$-martingale:\footnote{Note that this is a ``property" rather than a ``hypothesis" but we will stick to the historical terminology introduced in \cite{bremaud1978changes}.}
$$\cM^2\deln{\bbP,\del{\cM_t}_{t \in \sbr{0,1}}} \subseteq \cM^2\deln{\bbP,\del{\cN_t}_{t \in \sbr{0,1}}}. $$
It was noticed in \cite{causalopttptacciaio} that causality of $\pi \in \Pi\del{\eta,\nu}$ is equivalent to the H-hypothesis for the filtrations $\deln{\cH_t \otimes \cbrn{\emptyset, \cW^d}}_{t\in \sbr{0,1}}$ and $\del{\cH^\eta_t \otimes \cH_t}_{t\in \sbr{0,1}}$ on $\deln{\cW^d \times \cW^d, \cF_1 \otimes \cF_1, \pi}$. Since this is a key property that we will use repeatedly throughout the article, we state this as a theorem:

\begin{theorem}
\label{thm:equivctdscausality}
A coupling $\pi \in \Pi\del{\eta,\nu}$ is a causal coupling if and only if the filtrations $\deln{\cH^\eta_t \otimes \cbrn{\emptyset, \cW^d}}_{t\in \sbr{0,1}}$ and $\del{\del{\cH_t \otimes \cH_t}^\pi}_{t+ \in \sbr{0,1}}$ on $\deln{\cW^d \times \cW^d, \cF_1 \otimes \cF_1, \pi}$ 
satisfy the  $H$-hypothesis:
  \begin{equation}
        \label{eqn:hhypothesis}
        \cM^2\deln{\pi, \deln{\cH^\eta_t \otimes \cbrn{\emptyset, \cW^d} }_{t\in [0,1]}} \subseteq \cM^2\deln{\pi, \del{\del{\cH_t \otimes \cH_t}^\pi}_{t+ \in \sbr{0,1}}}.
    \end{equation} 
\end{theorem}
\begin{proof}
\label{proof:thmequivctdcausality}

From \cite[Theorem $6$]{bremaud1978changes} or \cite[Remark $2.3$]{causalopttptacciaio}, $\pi \in \Pi_c\del{\eta,\nu}$ if and only if   the H-hypothesis holds between filtrations $\deln{\cH_t \otimes \cbrn{\emptyset, \cW^d}}_{t\in \sbr{0,1}}$ and $\del{\cH^\eta_t \otimes \cH_t}_{t\in \sbr{0,1}}$ on $\deln{\cW^d \times \cW^d, \cF_1 \otimes \cF_1, \pi}$. Hence, it is clear that $ii. \implies i.$ since $\cH_t^\eta \otimes \cH_t \subseteq \cap_{\epsilon > 0} (\cH_{t + \epsilon} \otimes \cH_{t + \epsilon}^{\pi})$ for all $t \in \sbr{0,1}$. 

To prove $i. \implies ii.$, let $\del{M_t}_{t \in \sbr{0,1}}$ be a $\deln{\pi, \deln{\cH^\eta_t \otimes \cbrn{\emptyset, \cW^d}}}$-martingale, which we can assume without loss of generality to be c\'adl\'ag. It is straightforward to verify that $\del{M_t}_{t \in \sbr{0,1}}$ is a $\deln{\pi, \deln{\cH^\eta_t \otimes \cH_t}}$-martingale by causality of $\pi$ in light of \cite[Theorem $6$]{bremaud1978changes} or \cite[Remark $2.3$]{causalopttptacciaio}. Since $\del{M_t}_{t \in \sbr{0,1}}$ is c\`adl\`ag, it is even a $\deln{\pi,\del{\del{\cH_t \otimes \cH_t}^\pi}_{t+ \in \sbr{0,1}}}$-martingale by \cite[Lemma $67.10$]{rogers2000diffusions} as required.
\end{proof}
\begin{remark}
    Note that the filtrations in the above theorem are   complete and right-continuous. In particular, martingales have c\`adl\`ag modifications and up to standard localization techniques, it is sufficient to consider bounded martingales for checking the $H$-hypothesis. One can obtain many other equivalent conditions to causality as a result of \cite[Theorem $3$]{bremaud1978changes}.
\end{remark}

\section{Causal Transport  between weak solutions of stochastic differential equations}
\label{section:causaltptsbetweensdes}
\subsection{Bicausal couplings of Wiener measures}

We first examine bicausal couplings of (d-dimensional) Wiener measures.
\begin{theorem}[Bicausal couplings of Wiener measures]
  \label{thm:bicausaltWienercouplings}
    $\pi \in \Pi_{bc}\deln{\mathbb{W}^d,\mathbb{W}^d }$ if and only if  there exists a stochastic basis $\del{\Omega, \mathcal{G}, \del{\mathcal{G}_t}_{t \in \sbr{0,1}}, \mathbb{P}}$ 
and two $d$-dimensional $\del{\mathcal{G}_t}_{t \in \sbr{0,1}}$-Brownian motions $B  $ and $\tilde{B}  $
          such that   $\pi$ is the joint law of  $(B,\tilde{B})$.\\
              In particular, there exists a $\del{\mathcal{G}_t}_{t \in \sbr{0,1}}$-progressively measurable  process $\rho : \sbr{0,1} \times \Omega\rightarrow \mathcal{C}^d$ such that \[
  \forall t \in \sbr{0,1} \;\qquad  [B, \tilde{B}]_t = \int_0^t \rho_s ds, \;  \pi\text{-a.s.}.
    \]
\end{theorem}
As this is a special case of Theorem \ref{thm:bicausaltptplanbetweenSDEs}, we defer the proof to Section \ref{subsection:proofs}.

\subsection{Bicausal couplings between solutions of SDEs}

We now provide a  characterisation of bicausal couplings and Monge transports between weak solutions of SDEs. We first present the main results, Theorem \ref{thm:bicausaltptplanbetweenSDEs} and \ref{thm:bicausaltptmapsbetweenSDEsdegen}, and their corollaries and defer technical proofs to Section \ref{subsection:proofs}. Let \[
\mathscr{A}^{d,d'} := \cbr{\alpha : \sbr{0,1} \times \mathcal{W}^d \rightarrow \mathbb{R}^{d \times d'} : \alpha \text{ is } \del{\mathcal{F}_t}_{t \in \sbr{0,1}} \text{-progressively measurable} }
\] for $d, d' \geq 1$.  We consider the SDEs \begin{gather}
\label{eqn:firstsde}
dZ_t = b\del{t,Z}dt + \sigma\del{t,Z} dB_t, \quad Z_0 = z\\
\label{eqn:2ndsde}
d\tilde{Z}_t = \overline{b}\deln{t,\tilde{Z}} dt + \overline{\sigma}\deln{t,\tilde{Z}} d\tilde{B}_t, \quad \tilde{Z}_0 = \tilde{z}.
\end{gather}
where $\sigma,\overline{\sigma} \in \mathscr{A}^{d,d}$ and $b, \overline{b} \in \mathscr{A}^{d,1}$ are such that
\begin{assumption}
    \label{assumption:weakuniquenessassumption}
    There exists a unique weak solution  $\mu^{\sigma,b}_z \in \mathcal{P}\deln{\mathcal{W}^d}$ (resp. 
    $\mu^{\overline{\sigma}, \overline{b}}_{\tilde{z}} \in \mathcal{P}\deln{\mathcal{W}^d}$) for \eqref{eqn:firstsde} (resp. for \eqref{eqn:2ndsde}).
\end{assumption}

The following theorem is a characterization of the set of bicausal couplings between the solutions of SDEs. It generalizes \cite[Proposition $2.2$]{backhoffveraguas2024adaptedwassersteindistancelaws} to the case where $d > 1$ and the SDEs may not necessarily have strong solutions. 

\begin{theorem}
    \label{thm:bicausaltptplanbetweenSDEs}
    $\pi \in \Pi_{bc}\deln{\mu^{\sigma,b}_z, \mu^{\overline{\sigma}, \overline{b}}_{\tilde{z}}}$ if and only if  there exists a stochastic basis $\del{\Omega, \mathcal{G}, \del{\mathcal{G}_t}_{t \in \sbr{0,1}}, \mathbb{P}}$ 
and two $d$-dimensional $\del{\mathcal{G}_t}_{t \in \sbr{0,1}}$-Brownian motions $B = \del{B_t}_{t \in \sbr{0,1}}$ and $\tilde{B} = \deln{\tilde{B}_t}_{t \in \sbr{0,1}}$
        and $\del{\mathcal{G}_t}_{t \in \sbr{0,1}}$-adapted processes $Z = \deln{Z_t}_{t \in \sbr{0,1}}$, $\tilde{Z} = \deln{\tilde{Z}_t}_{t \in \sbr{0,1}}$ such that 
        \begin{itemize}
            \item $\deln{Z,B}$ and  $\deln{\tilde{Z},\tilde{B}}$ satisfy  \eqref{eqn:firstsde} and \eqref{eqn:2ndsde} respectively, and
            \item $\pi$ is the joint law of  $(Z,\tilde{Z})$: $\pi = \deln{Z,\tilde{Z}}_{\#}\mathbb{P}$.
        \end{itemize}
        In particular, the joint law of $(B,\tilde{B})$ under $\mathbb{P}$ is a bicausal Wiener coupling. 
\end{theorem}
In other words, all bicausal couplings may be obtained by coupling the Brownian motions in \eqref{eqn:firstsde} and \eqref{eqn:2ndsde}, using a bicausal coupling of Wiener measures.

Consider the case of bicausal couplings between  $d$-dimensional Wiener measures. Theorem \ref{thm:bicausaltptplanbetweenSDEs} gives that $\pi \in \Pi_{bc}\deln{\mathbb{W}^d, \mathbb{W}^d}$ if and only if $\pi$ is the joint law of two processes which are Brownian motions with respect to a \emph{common} filtration. A   consequence of this observation is that all bicausal couplings between SDEs with strong, pathwise unique  solutions are obtained by taking the image of bicausal Wiener couplings in $\Pi_{bc}\deln{\mathbb{W}^d, \mathbb{W}^d}$ under the respective Ito maps (Figure \ref{fig:bicausalcouplingcomposition}). 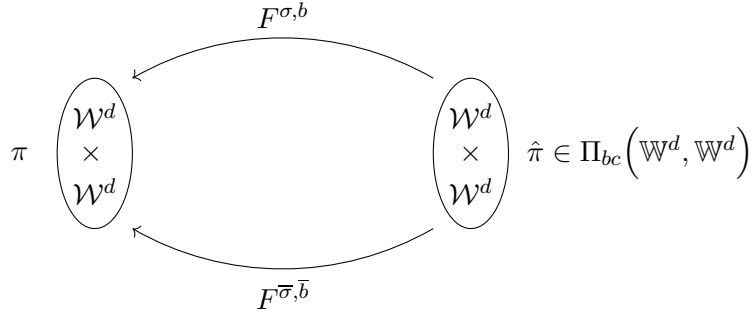
\begin{figure}
    \centering
    \begin{tikzpicture}
          \draw[->] (0,0.5) arc (60:120:4cm) node[midway,above]{$F^{\sigma,b}$};
          \draw[<-] (-4,-1.5) arc (240:300:4cm) node[midway,below]{$F^{\overline{\sigma},\overline{b}}$};
          \node[] at (-4.55,-0.5) {$\times$};
          \node[] at (-4.5,0) {$\mathcal{W}^d$};
          \node[] at (-4.5,-1) {$\mathcal{W}^d$};
          \node[] at (-5.5,-0.5) {$
          \pi$};
          \draw (-4.5,-0.5) ellipse (0.5cm and 1cm);
          \node[] at (0.5,-0.5) {$\times$};
          \node[] at (0.5,0) {$\mathcal{W}^d$};
          \node[] at (0.5,-1) {$\mathcal{W}^d$};
          \draw (0.5,-0.5) ellipse (0.5cm and 1cm);
          \node[] at (2.75,-0.5) {$\hat{\pi} \in \Pi_{bc}\del{\mathbb{W}^d, \mathbb{W}^d}$};
      \end{tikzpicture}
    \caption{ Bicausal couplings $\pi \in \Pi_{bc}\deln{\mu^{\sigma,b}_z, \mu^{\overline{\sigma}, \overline{b}}_{\tilde{z}}}$ between (strong) solutions of   SDEs  may be constructed by pushing forward bicausal Wiener couplings $\hat{\pi} \in \Pi_{bc}\deln{\bbW^d,\bbW^d}$ with the Ito maps $\deln{F^{\sigma,b}, F^{\overline{\sigma}, \overline{b}}}$.}
    \label{fig:bicausalcouplingcomposition}
\end{figure}
\begin{corollary}
    \label{corollary:pushforwardofbicausalwiener}
    Assume the SDEs \eqref{eqn:firstsde} and \eqref{eqn:2ndsde} have strong solutions satisfying  pathwise uniqueness, given by the Ito maps $F^{\sigma,b} : \deln{\mathcal{W}^d, \cF_1} \rightarrow \deln{\mathcal{W}^d, \cF_1}$ (resp. $F^{\overline{\sigma}, \overline{b}})$. Then all 
    bicausal couplings $\pi \in \Pi_{bc}\deln{\mu^{\sigma,b}_z, \mu^{\overline{\sigma}, \overline{b}}_{\tilde{z}}}$ between the  solutions of the SDEs  may be obtained by pushing forward bicausal Wiener couplings $\hat{\pi} \in \Pi_{bc}\deln{\bbW^d,\bbW^d}$ with the Ito maps $\deln{F^{\sigma,b}, F^{\overline{\sigma}, \overline{b}}}$:
    \[
    \Pi_{bc}\deln{\mu^{\sigma,b}_z, \mu^{\overline{\sigma}, \overline{b}}_{\tilde{z}}} = \cbr{\del{F^{\sigma,b}, F^{\overline{\sigma}, \overline{b}}}_{\#}\hat{\pi} : \hat{\pi} \in \Pi_{bc}\deln{\mathbb{W}^d, \mathbb{W}^d}}.
    \] 
\end{corollary}
\begin{proof} The construction is shown in Figure \ref{fig:bicausalcouplingcomposition}).
    Note that the assumption of unique, strong solutions automatically implies weak uniqueness so that Assumption \ref{assumption:weakuniquenessassumption} holds and we can use Theorem \ref{thm:bicausaltptplanbetweenSDEs}, which yields the desired result.
\end{proof}

The preceding corollary allows us to deduce the properties of $ \Pi_{bc}\deln{\mu^{\sigma,b}_z, \mu^{\overline{\sigma}, \overline{b}}_{\tilde{z}}}$ from the study of $\Pi_{bc}\deln{\mathbb{W}^d, \mathbb{W}^d}$, which we will use in Section \ref{section:applications}. 

\subsection{Characterisation of bicausal Monge transports}

For bicausal Monge maps, we first note that on $\deln{\cW^d, \cF_1, \mu^{\sigma,b}_z}$, the canonical process $\del{X_t}_{t \in \sbr{0,1}}$ is a $\deln{\mu^{\sigma,b}_z,\del{\cH_t}}$-semimartingale with decomposition \begin{equation}
    \label{eqn:canonicalsemimartdecomp}
    X_t = z + \int_0^t b\del{s,X} ds + M^{\mu^{\sigma,b}_z}_t \quad \mu^{\sigma,b}_z\text{-a.s. } t \in \sbr{0,1},
\end{equation} where $\deln{M^{\mu^{\sigma,b}_z}_t}$ is a $\deln{\mu^{\sigma,b}_z,\del{\cH_t}}$-local martingale with quadratic variation \begin{equation}
    \label{eqn:qvofmartcanonical}
    \sbr{M^{\mu^{\sigma,b}_z},M^{\mu^{\sigma,b}_z}}_t = \int_0^t \sigma\del{s,X}\sigma^*\del{s,X} ds \quad \mu^{\sigma,b}_z\text{-a.s. }t \in \sbr{0,1}.
\end{equation} We then have the following characterization:

\begin{theorem}[Characterisation of bicausal Monge transports]
\label{thm:bicausaltptmapsbetweenSDEsdegen}
    Let $\eta := \mu^{\sigma,b}_z $ and $\nu := \mu^{\overline{\sigma}, \overline{b}}_{\tilde{z}}$.
   For $T : \cW^d \rightarrow \cW^d$, define  $T_t\del{\omega} := T\del{\omega}\del{t} \in \mathbb{R}^{d}$ for $t \in \sbr{0,1}$. $T$ is a bicausal Monge map from $\eta$ to $\nu$ if and only if\;\footnote{We thank Y. Jiang for the sufficiency part of this theorem.} there exists some $\del{\mathcal{H}^\eta_{t}}_{t \in \sbr{0,1}}$-progressively measurable process $Q : \sbr{0,1} \times \cW^d \rightarrow \cO^d$ 
     such that $\del{T_t}_{t \in \sbr{0,1}}$ is a $\deln{\eta, \del{\cH^\eta_t}_{t \in \sbr{0,1}}}$-semimartingale satisfying 
     \begin{equation}
     \label{eqn:bicausalmongemapeqn}
     T_t = \tilde{z} + \int_0^t \overline{b}\del{s,T} ds + \int_0^t \overline{\sigma}\del{s,T} Q_s \sigma^\dag\del{s,X} dM^\eta_s 
     \; \eta\text{-a.s.}
     \end{equation}
      and
      \begin{equation}
      \label{eqn:Qstructureeqn}
     \overline{\sigma}\del{t,T} Q_t = \overline{\sigma}\del{t,T} Q_t \sigma^\dag\del{t,X}\sigma\del{t,X} \qquad d\eta\otimes dt\quad -{\rm a.e. }  
      \end{equation}
\end{theorem}

Since $\sigma^\dag\del{s,X} \sigma\del{s,X}$ is an orthogonal projection onto the orthogonal complement of $\Ker \sigma\del{s,X}$ (see \cite[Section $5.5.4$]{golub2013matrix}), \[
\eqref{eqn:Qstructureeqn} \iff Q_s \Ker\del{\sigma\del{s,X}} \subseteq \Ker \del{\overline{\sigma}\del{s,T}}.
\]
\begin{proof}
    Let $P_s := \sigma^\dag\del{s,X}\sigma\del{s,X}$. Then, $    \overline{\sigma}\del{s,T}Q_s = \overline{\sigma}\del{s,T}Q_s P_s + \overline{\sigma}\del{s,T}Q_s\del{\text{Id}_d - P_s}$. Hence, $\eqref{eqn:Qstructureeqn} \iff \overline{\sigma}\del{s,T}Q_s\del{\text{Id}_d - P_s} = 0$. Since $\text{Id}_d - P_s$ is the orthogonal projection onto $\Ker \sigma\del{s,X}$, the statement follows.
\end{proof}
Equation \eqref{eqn:Qstructureeqn}
enables to identify situations where  bicausal Monge maps cannot exist:
\begin{corollary}
    \label{corollary:nobicausalmongemaps}
    Under the assumptions of Theorem \ref{thm:bicausaltptmapsbetweenSDEsdegen} 
    , if for all $s \in \sbr{0,1}$, \[
    \max_{\omega \in \cW^d} \dim \deln{\Ker\del{\overline{\sigma}\del{s, \omega}}} < \min_{\omega \in \cW^d} \dim \deln{\Ker\del{\sigma\del{s,\omega}}},
    \] there exists no bicausal Monge maps from $\eta$ to $\nu$.
\end{corollary}

 The bicausal Monge transports which leave the Wiener measure invariant are exactly (causal) local rotations, which reflect the rotational gauge symmetry of the Wiener measure: \begin{corollary}[Bicausal Monge transports between Wiener measures]
    \label{corollary:bicausalMongemapbetweenWienermeasures}
    The set of all bicausal Monge transports from $\bbW^d$ to $\bbW^d$ is given by \[
    \cT_{bc}(\bbW^d,\bbW^d)=(\cbr{\int_0^\cdot Q_s dX_s,\quad  Q:[0,1]\times {\cal W}^d\mapsto \cO^d \text{ is    $\deln{\cH^{\bbW^d}_t}_{t \in \sbr{0,1}}$-progressively measurable }}.
    \]
\end{corollary}

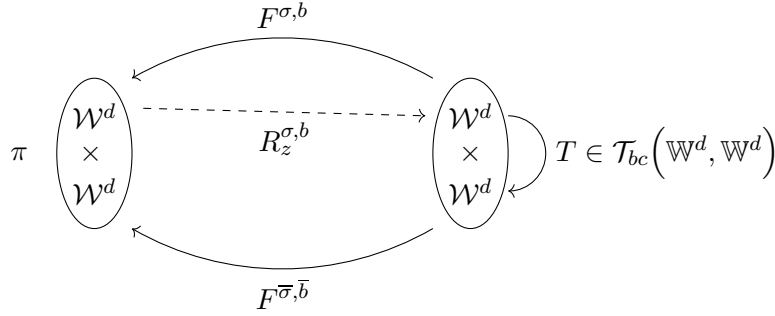
\begin{figure}
    \centering
    \begin{tikzpicture}
          \draw[->] (0,0.5) arc (60:120:4cm) node[midway,above]{$F^{\sigma,b}$};
          \draw[->,dashed] (-3.85,0.1) -- (-0.1,0) node[midway,below]{$R^{\sigma,b}_z$};
          \draw[<-] (-4,-1.5) arc (240:300:4cm) node[midway,below]{$F^{\overline{\sigma},\overline{b}}$};
          \node[] at (-4.55,-0.5) {$\times$};
          \node[] at (-4.5,0) {$\mathcal{W}^d$};
          \node[] at (-4.5,-1) {$\mathcal{W}^d$};
          \node[] at (-5.5,-0.5) {$
          \pi$};
          \draw (-4.5,-0.5) ellipse (0.5cm and 1cm);
          \node[] at (0.5,-0.5) {$\times$};
          \node[] at (0.5,0) {$\mathcal{W}^d$};
          \node[] at (0.5,-1) {$\mathcal{W}^d$};
          \draw[->] (1,0) arc (90:-90:0.5cm) node[midway,right] {$T \in \mathcal{T}_{bc}\del{\bbW^d, \bbW^d}$};
          \draw (0.5,-0.5) ellipse (0.5cm and 1cm);
      \end{tikzpicture}
    \caption{Any bi-causal Monge transport $\pi \in \mathcal{T}_{bc}\deln{\mu^{\sigma,b}_z, \mu^{\overline{\sigma}, \overline{b}}_{\tilde{z}}}$ can be constructed by pushing $\mu^{\sigma,b}_z$ along $R^{\sigma,b}_z$, then $T$ and finally, $F^{\overline{\sigma}, \overline{b}}_{\tilde{z}}$.}
    \label{fig:bicausalmongecomposition}
\end{figure}

More generally, if $\sigma$ is invertible, then $\sigma^\dag\del{s,X} = \sigma^{-1}\del{s,X}$ and \eqref{eqn:Qstructureeqn} holds. We may now describe the general structure of a bicausal Monge transport between strong solutions of SDEs as follows (see Figure \ref{fig:bicausalmongecomposition}): \begin{corollary}[Bicausal Monge transport between strong solutions of SDEs]
    \label{corollary:strongsolnsdebicausalmongemap}
   Assume the SDEs \eqref{eqn:firstsde} and \eqref{eqn:2ndsde} have strong solutions satisfying  pathwise uniqueness, given by the Ito maps $F^{\sigma,b} : \deln{\mathcal{W}^d, \cF_1} \rightarrow \deln{\mathcal{W}^d, \cF_1}$ (resp. $F^{\overline{\sigma}, \overline{b}})$. If $\sigma\del{s, \omega}$ is invertible for all $\del{s,\omega}$, then the set of all bicausal Monge transports from $\mu^{\sigma,b}_z$ to $\mu^{\overline{\sigma}, \overline{b}}_{\tilde{z}}$ is   \[
    \cbr{F^{\overline{\sigma}, \overline{b}} \circ T \circ R^{\sigma,b}_z,\qquad T\in \cT_{bc}(\bbW^d,\bbW^d)\ },
    \] where 
    $R^{\sigma,b}_z : \cW^d \rightarrow \cW^d$ is defined by $$R^{\sigma,b}_z\del{X} = \int_0^\cdot \sigma^{-1}\del{s,X} dM^\eta_s \qquad \eta\-{\rm a.s.\  and}\quad M^\eta_\cdot := X_\cdot - z - \int_0^\cdot b(s,X) ds.$$
\end{corollary}

\begin{proof}
    Let $\eta := \mu^{\sigma,b}_z$ and $\deln{\tilde{\cF}_t}_{t \in \sbr{0,1}}$ be the right-continuous version of the filtration generated by $R^{\sigma,b}_z = \deln{\deln{R^{\sigma,b}_z}_t}_{t \in \sbr{0,1}}$, where $\deln{R^{\sigma,b}_z}_t := \int_0^t \sigma^{-1}\del{s,X} dM^\eta_s$. We claim that $\deln{\tilde{\cF}^\eta_t}_{t \in \sbr{0,1}} = \del{\cH^\eta_t}_{t \in \sbr{0,1}}$. It is clear that $\tilde{\cF}^\eta_t \subseteq \cH^\eta_t$ for all $t \in \sbr{0,1}$. To see the converse, note that $\deln{\deln{R^{\sigma,b}_z}_t}_{t \in \sbr{0,1}}$ is a $\deln{\tilde{\cF}^\eta_t}_{t \in \sbr{0,1}}$-Brownian motion by L\'evy's characterization of Brownian motion. Furthermore, by definition of $\deln{\deln{R^{\sigma,b}_z}_t}_{t \in \sbr{0,1}}$ and \eqref{eqn:canonicalsemimartdecomp}, we have \[
    X_t = z + \int_0^t b\del{s,X} ds + \int_0^t \sigma\del{s,X} d\deln{R^{\sigma,b}_z}_s, \; t \in \sbr{0,1} \; \eta\text{-a.s.} . 
    \] By pathwise uniqueness, $X$ must be adapted to the filtration generated by Brownian motion $R^{\sigma,b}_z$ so that $\cH^\eta_t \subseteq \tilde{\cF}^\eta_t$ for all $t \in \sbr{0,1}$. 
    
    Let $\hat{T}$ be a bicausal Monge map from $\mu^{\sigma,b}_z$ to $\mu^{\overline{\sigma}, \overline{b}}_{\tilde{z}}$. Then, by Theorem \ref{thm:bicausaltptmapsbetweenSDEsdegen}, there exists some $\deln{\cH_t^{\eta}}_{t \in \sbr{0,1}}$-progressively measurable $\cO^d$-valued process $\del{Q_t}_{t \in \sbr{0,1}}$ such that \eqref{eqn:bicausalmongemapeqn} and \eqref{eqn:Qstructureeqn} holds and by pathwise uniqueness, it is clear that $\hat{T} = F^{\overline{\sigma}, \overline{b}}\deln{\int_0^t Q_s d{\deln{R^{\sigma,b}_z}}_s}$. But since $\deln{\tilde{\cF}^\eta_t}_{t \in \sbr{0,1}} = \deln{\cH^\eta_t}_{t \in \sbr{0,1}}$, then there exists a $\deln{\cH_t^{\bbW^d}}_{t \in \sbr{0,1}}$-progressively measurable $\cO^d$-valued process process $\deln{\tilde{Q}_t}$ such that $\tilde{Q}_t \circ R = Q_t$ $\eta$-a.s. for all $t \in \sbr{0,1}$. The map $X \mapsto T\del{X} := \int_0^\cdot \tilde{Q}_s dX_s$ is a bicausal Monge map from $\bbW^d$ to $\bbW^d$ by Corollary \ref{corollary:bicausalMongemapbetweenWienermeasures} and it can be readily verified that $\eta$-a.s. $T\deln{R^{\sigma,b}_z} = \int_0^\cdot Q_s d\deln{R^{\sigma,b}_z}_s$. Hence, $\hat{T} = F^{\overline{\sigma}, \overline{b}}\deln{T\deln{R^{\sigma,b}_z}}$ as required.

    Finally, since \eqref{eqn:Qstructureeqn} is always satisfied by assumption on $\sigma$, it is clear to see that for any $T \in \cT_{bc}\deln{\bbW^d, \bbW^d}$, the map $ F^{\overline{\sigma}, \overline{b}}\deln{T\deln{R^{\sigma,b}_z}}$ is a bicausal Monge map from $\mu^{\sigma,b}_z$ to $\mu^{\overline{\sigma}, \overline{b}}_{\tilde{z}}$ using Theorem \ref{thm:bicausaltptmapsbetweenSDEsdegen}. 
\end{proof}
Corollary \ref{corollary:strongsolnsdebicausalmongemap} is the case where $\Ker\del{\sigma\del{s,X}} = \cbrn{0} \subseteq \Ker\del{\overline{\sigma}\del{s,T}}$ so that \eqref{eqn:Qstructureeqn} is \emph{always} satisfied, while Corollary \ref{corollary:nobicausalmongemaps} is the case where \eqref{eqn:Qstructureeqn} is \emph{never} satisfied. One can construct intermediate situations via conditions on the diffusion coefficient $\overline{\sigma}$. For example, \begin{corollary}
    \label{corollary:existenceofbicausalmongemapssimpleeg}
    Let $\sigma, \overline{\sigma} \in \mathscr{A}^{d, d}$ be  such that for all $t \in \sbr{0,1}$, $\omega, \omega' \in \cW^d$, \[
    \envert{\overline{\sigma}\del{t, \omega} - \overline{\sigma}\del{t, \omega'}} \leq L \norm{\omega - \omega'}_\infty \text{ and } \int_0^1 \envert{\overline{\sigma}\del{s,0}}^2 ds < \infty.\]
    for some   $L > 0$. If $\Ker \del{\overline{\sigma}}$ is  deterministic   and for all $s \in \sbr{0,1}$, \[
    \max_{\omega \in E} \text{dim} \del{\Ker\del{\sigma\del{s,\omega}}} \leq \text{dim}\del{\Ker\del{\overline{\sigma}}\del{s}},
    \] there exists a bicausal Monge transport from $\mu^{\sigma,0}_z$ to $\mu^{\overline{\sigma},0}_z$.
\end{corollary}

\begin{proof}
   We  construct an $\cO^d$-valued process $\del{Q_t}_{t \in \sbr{0,1}}$ such that for all $s \in \sbr{0,1}, \omega \in \cW^d$, \[
   Q_s \text{Ker}\del{\sigma\del{s,\omega}} \subseteq \text{Ker}\del{\overline{\sigma}\del{s}}\]
   so that \eqref{eqn:Qstructureeqn} holds. Let $\overline{m}\del{s} := \text{dim}\deln{\text{Ker}\del{\overline{\sigma}\del{s}}}$ and $m\del{s,\omega} := \text{dim}\deln{\text{Ker}\del{\sigma\del{s,\omega}}}$ so that $m \leq \overline{m}$. Letting $\overline{\sigma} = \overline{U} \overline{\Sigma} \overline{V}^*$ be the singular value decomposition (SVD) of $\overline{\sigma}$, we see that for fixed $s \in \sbr{0,1}$, $d - \overline{m} = \text{max}\cbr{k : \overline{\Sigma}_{kk} > 0}$ (without loss of generality, the diagonal entries of $\overline{\Sigma}$ are assumed to be decreasing). Hence, writing $\overline{v}^*_i$ as the $i^{th}$-column of $\overline{V}^*$ and $\overline{\alpha} := d - \overline{m}$, then $\text{Ker}\del{\overline{\sigma}\del{s,\omega}} = \text{Span}\deln{\cbr{\overline{v}_i^*}_{i = \overline{\alpha} + 1}^d}$. Since $\text{Ker}\del{\overline{\sigma}\del{s,\omega}}$ is deterministic, up to replacing the last $\overline{m}$ columns of $\overline{V}^*$, we assume that $\cbr{\overline{v}^*_i}_{i = \overline{\alpha} + 1}^d$ are vectors depending on time, but not on $\omega$. Hence, let $\overline{O}^*_s$ be any time-dependent orthogonal matrix whose last $\overline{m}$ columns are $\cbr{\overline{v}^*_i}_{i = \overline{\alpha} + 1}^d$. Letting $\sigma = U \Sigma V^*$ be the SVD of $\sigma$, one easily verifies that $Q\del{s,\omega} := \overline{O}^*_s V\del{s,\omega}$ is the desired $Q$ noting that the kernel of $\sigma\del{s,\omega}$ is given by the last $m$ vectors of $V_s^*$. 
Picard iteration   applied to \eqref{eqn:bicausalmongemapeqn}  yields the existence of a bicausal Monge transport.
\end{proof}

\subsection{Proof of Theorem \ref{thm:bicausaltptplanbetweenSDEs} and \ref{thm:bicausaltptmapsbetweenSDEsdegen}}
\label{subsection:proofs}

\subsubsection{Martingale Representation Theorem}

A crucial ingredient in the proof of our results is a  martingale representation property   \cite[Theorem $2$]{ustunel2019martingale} applicable to   SDEs with unique weak solutions (see also \cite[Ch. 7]{bally2016stochastic}).

\begin{prop}
        \label{prop:degenMRPoncanonicalspace}
    On $\del{\cW^d, \cF_1, \del{\cH^\eta_t}_{t \in \sbr{0,1}}, \eta}$, where $\eta := \mu^{\sigma,b}_z$, we have that for any $F \in L^2\del{\eta, \cF_1}$, there exists a $\del{\cH^\eta_{t}}_{t \in \sbr{0,1}}$-progressively measurable process $\xi : \sbr{0,1} \times \cW^d \rightarrow \mathbb{R}^d$ such that \[
    \bbE^\eta\del{\int_0^1 \norm{\xi_s}^2_{\mathbb{R}^d} ds} < \infty
    \] and \[
    F = \int_0^1 \del{\xi_s, dL_s}_{\bbR^d} = \int_0^1 \sum_{i = 1}^d \xi^i_s dL^i_s \quad \eta\text{-a.s.},
    \] where $L_\cdot = \int_0^\cdot \sigma^\dag\del{s,X} dM^\eta_s$. In particular, any square-integrable $\deln{\eta,\del{\cH_{t}^\eta}_{t \in \sbr{0,1}}}$-martingale $\del{M_t}_{t \in \sbr{0,1}}$ has representation \[
    M_t = \int_0^t \sum_{i = 1}^d \xi^i_s dL^i_s, \quad \forall t \in \sbr{0,1} \quad \eta\text{-a.s.}. 
    \]
\end{prop}

\begin{proof}
Consider the filtered probability space \[
\del{\Omega, \cG, \del{\cG_t}_{t \in \sbr{0,1}}, \bbP} := \del{\cW^d \times \cW^d, \cF_1 \otimes \cF_1, \deln{\deln{\cH_t \otimes \cH_t}^{\eta \otimes \bbW^d}}_{t + \in \sbr{0,1}}, \eta \otimes \bbW^d}.
\] Since $X_{\#}\bbP = \eta$, we can identify $L^2\del{\eta,\cF_1}$ with $L^2\deln{\bbP, \cF_1 \otimes \cbrn{\emptyset, \cW^d}}$ by the Doob-Dynkin lemma (see \cite{kallenberg1997foundations}) and we no longer distinguish these spaces. It is straightforward to verify that $\bbP \in \Pi_{bc}\deln{\eta, \bbW^d}$. \eqref{eqn:canonicalsemimartdecomp} and Theorem \ref{thm:equivctdscausality} implies that $\del{X_t}_{t \in \sbr{0,1}}$ is a $\deln{\bbP,\del{\cG_t}}$-semimartingale such that $\del{X_t - z - \int_0^t b\del{s,X} ds}_{t \in \sbr{0,1}}$ is a $\deln{\bbP,\del{\cG_t}}$-local martingale and \[
\sbr{X,X}_t = \int_0^t \sigma\del{s,X} \sigma^*\del{s,X} ds \quad \bbP\text{-a.s. } t \in \sbr{0,1}
\] and $\del{Y_t}_{t \in \sbr{0,1}}$ is a $\deln{\bbP,\del{\cG_t}}$-Brownian motion independent from $\del{X_t}_{t \in \sbr{0,1}}$. Following the proof of \cite[Theorem $7.1'$]{ikeda2014stochastic}, we define \begin{equation}
    \label{eqn:constructedbrownian}
    B_t := \int_0^t \sigma^\dag\del{s,X} dM_s + \int_0^t \deln{\text{Id}_d - \sigma^\dag\del{s,X}\sigma\del{s,X}} dY_s, 
\end{equation} where $M_t := X_t - z - \int_0^t b\del{s,X}ds$. Then, $\del{B_t}_{t \in \sbr{0,1}}$ is a $\deln{\bbP,\del{\cG_t}}$-Brownian motion by L\'evy's characterization of Brownian motion and $M_t = \int_0^t \sigma\del{s,X} dB_s$ for all $t \in \sbr{0,1}$ $\bbP$-a.s.. That is, the pair $\del{X,B}$ satisfies SDE \eqref{eqn:firstsde}. Since $\sigma^\dag\del{s,X} \sigma\del{s,X}$ is an orthogonal projection onto the range of $\sigma^*\del{s,X}$ (see \cite[Section $5.5.4$]{golub2013matrix}), we have by \cite[Theorem $2$]{ustunel2019martingale} that for any $F \in L^2\del{\eta, \cF_1}$, there exists a $\deln{\cH^\eta_t \otimes \cbrn{\emptyset, \cW^d}}_{t \in \sbr{0,1}}
$-progressively measurable process $\del{\xi_t}$ such that \[
\bbE^\bbP \del{\int_0^1 \norm{\xi_s}^2_{\bbR^d} ds } < \infty 
\] and \[
F = \int_0^1 \del{\xi_s, \sigma^\dag\del{s,X} \sigma\del{s,X} dB_s}_{\bbR^d} \; \bbP\text{-a.s..}
\] Using \eqref{eqn:constructedbrownian}, one can verify $\sigma^\dag\del{s,X} \sigma\del{s,X} dB_s = \sigma^\dag\del{s,X} dM_s$, from which we obtain the desired result.
\end{proof}

\subsubsection{Technical proofs}

\begin{proof}[Proof of Theorem \ref{thm:bicausaltptplanbetweenSDEs}]

    Denote $\eta := \mu^{\sigma,b}_z $ and $\nu := \mu^{\overline{\sigma}, \overline{b}}_{\tilde{z}}$.
     We first prove sufficiency, that is, $\pi = \deln{Z, \tilde{Z}}_{\#}\bbP \in \Pi_{bc}\del{\eta,\nu}$. By Assumption \ref{assumption:weakuniquenessassumption}, it is clear $\pi \in \Pi\del{\eta,\nu}$ and it remains to show bicausality of $\pi$. We begin by verifying that $X$ is a $\deln{\pi, \deln{\del{\cH_t \otimes \cH_t}^\pi}_{t + \in \sbr{0,1}}}$-semimartingale, for which it is enough to show that $M_t := X_t - z - \int_0^t b\del{s,X}ds$ is a $\deln{\pi, \deln{\del{\cH_t \otimes \cH_t}^\pi}_{t + \in \sbr{0,1}}}$-local martingale. Up to  localization, we may assume $\del{M_t}_{t \in \sbr{0,1}}$ is bounded. Since $Z^{-1}\del{\cH_t} = \cap_{\epsilon > 0} \mathcal{F}^{Z}_{t + \epsilon}$, where $\deln{\mathcal{F}_t^{Z}}_{t \in \sbr{0,1}}$ is the filtration generated by $\deln{Z_t}_{t \in \sbr{0,1}}$, and $\del{\mathcal{G}_t}_{t \in \sbr{0,1}}$ is right-continuous by assumption, it follows that $Z^{-1}\del{\cH_t} \subseteq \mathcal{G}_t$. Similarly, $\tilde{Z}^{-1}\del{\cH_t} \subseteq \mathcal{G}_t$. Hence, we calculate for $0 \leq s < t \leq 1$, $A, B \in \cH_s$ and $\deln{\omega, \omega'} \mapsto \mathbbm{1}_{A \times B}\deln{\omega,\omega'}$, \begin{multline*}
    \mathbb{E}^\pi \del{M_t \mathbbm{1}_{A \times B}} = \mathbb{E}^{\mathbb{P}} \deln{\int_0^t \sigma\del{u,Z_u} dB_u \mathbbm{1}_{Z^{-1}\del{A} \cap \tilde{Z}^{-1}\del{B}}} \\
    = \mathbb{E}^{\mathbb{P}}\del{\int_0^s \sigma\del{u,Z_u} dB_u \mathbbm{1}_{Z^{-1}\del{A} \cap \tilde{Z}^{-1}\del{B}}} = \mathbb{E}^{\pi}\del{M_s \mathbbm{1}_{A \times B}}
    \end{multline*} since $\int_0^\cdot \sigma\deln{u,Z_u}dB_u$ is a $\deln{\bbP,\del{\mathcal{G}_t}}$-martingale. This shows that $\del{M_t}_{t \in \sbr{0,1}}$ is a $\deln{\pi,\del{\cH_t \otimes \cH_t}}$-martingale and even a $\deln{\pi,\deln{\del{\cH_t \otimes \cH_t}^\pi}_{t + \in \sbr{0,1}}}$-martingale by continuity of sample paths and \cite[Lemma $67.10$]{rogers2000diffusions}. By the stochastic integral representation of elements in $\cM^2\deln{\pi, \deln{\cH^\eta_t \otimes \cbrn{\emptyset, \cW^d}}}$ from Proposition \ref{prop:degenMRPoncanonicalspace} and Theorem \ref{thm:equivctdscausality} $ii.$, we conclude $\pi \in \Pi_c\del{\eta,\nu}$. A completely symmetric argument yields $\pi \in \Pi_{bc}\del{\eta,\nu}$ as required.

    To see necessity, let $\pi \in \Pi_{bc}\del{\eta,\nu}$. By Theorem \ref{thm:equivctdscausality}, $X$ and $Y$ are $\deln{\pi, \deln{\del{\cH_t \otimes \cH_t}^\pi}_{t + \in \sbr{0,1}}}$-semimartingales. Up to enlarging the probability space, we assume that there exists a Brownian motion $\hat{B}$ independent from $X$ and $Y$. It is easy to check that $X,Y$ and $\hat{B}$ remain semimartingales with respect to probability measure $\pi$ and the (complete, right-continuous) filtration generated by all three processes, denoted $\del{\cG_t}$. As in the proof of Proposition \ref{prop:degenMRPoncanonicalspace}, we construct $\del{\cG_t}$-Brownian motions \begin{gather*}
        B_t := \int_0^t \sigma^\dag\del{s,X} dM_s + \int_0^t \deln{\text{Id}_d - \sigma^\dag\del{s,X}\sigma\del{s,X}} d\hat{B}_s, \\
        \tilde{B}_t := \int_0^t \overline{\sigma}^\dag\del{s,Y} dN_s + \int_0^t \deln{\text{Id}_d - \overline{\sigma}^\dag\del{s,Y}\overline{\sigma}\del{s,Y}} d\hat{B}_s,
    \end{gather*} where $N_t := Y_t - \tilde{z} - \int_0^t \overline{b}\del{s,Y} ds$. Noting that $M_t = \int_0^t \sigma\del{s,X} dB_s$ and $N_t = \int_0^t \overline{\sigma}\del{s,Y} d\tilde{B}_s$ yields the desired.

    The final assertion regarding the correlation process $\rho$ is exactly \cite[Lemma $5$]{emery2005certain}.
\end{proof}

To prove Theorem \ref{thm:bicausaltptmapsbetweenSDEsdegen}, we require the following generalization of \cite[Proposition $4$]{causalopttptlassalle}.

\begin{lemma}
    \label{lemma:bicausalmongemapequivctd}
    Under the assumptions of Theorem \ref{thm:bicausaltptmapsbetweenSDEsdegen}, a measurable map $T : \cW^d \rightarrow \cW^d$ from $\eta$ to $\nu$, i.e. $T_{\#} \eta = \nu$, is a bicausal Monge map if and only if $\del{T_t}_{t \in \sbr{0,1}}$ is a $\deln{\eta,\del{\cH_t^\eta}_{t \in \sbr{0,1}}}$-semimartingale. 
\end{lemma}

\begin{proof}
    Suppose $T$ is a bicausal Monge map from $\eta$ to $\nu$, i.e. $\pi_T := \del{X_{\cdot},T}_{\#}\eta \in \Pi_{bc}\del{\eta,\nu}$. Causality readily implies that $\del{T_t}_{t \in \sbr{0,1}}$ is $\deln{\cH^\eta_t}_{t \in \sbr{0,1}}$-adapted. Since $T$ has law $\mu^{\sigma,b}_{\tilde{z}}$, then $\del{T_t}_{t \in \sbr{0,1}}$ is a semimartingale with respect to probability measure $\eta$ and its own filtration, with martingale part $N_t := T_t - \tilde{z} - \int_0^t \overline{b}\del{s,T} ds$. Up to standard localization methods, we assume $\del{N_t}_{t \in \sbr{0,1}}$ is bounded. Bicausality implies that $\del{N_t}_{t \in \sbr{0,1}}$ is a $\deln{\eta,\del{\cH^\eta_t}}$-martingale, which implies the desired. Indeed, for $0 \leq s \leq t \leq 1$ and $A \in \cH_s$, we have \begin{multline}
    \label{eqn:computationmongecausality}
        \bbE^{\eta}\del{N_t \mathbbm{1}_{A}} = \bbE^{\pi_T}\deln{\deln{Y_t - \tilde{z} - \int_0^t \overline{b}\del{u,Y} du} \mathbbm{1}_{A \times \cW^d}} \\
        = \bbE^{\pi_T}\deln{\deln{Y_s - \tilde{z} - \int_0^s \overline{b}\del{u,Y} du} \mathbbm{1}_{A \times \cW^d}} = \bbE^\eta\del{N_s \mathbbm{1}_A},
    \end{multline} where we have used causality of $R_{\#}\pi_T \in \Pi_{c}\del{\nu, \eta}$ and Theorem \ref{thm:equivctdscausality} $ii.$ in the second equality. 

    Conversely, if $\del{T_t}_{t \in \sbr{0,1}}$ is a $\deln{\eta,\del{\cH^\eta_t}}$-semimartingale, then it is $\deln{\cH^\eta_t}_{t \in \sbr{0,1}}$-adapted so that $\pi_T \in \Pi_c\del{\eta,\nu}$. Furthermore, it holds that the first and last term in \eqref{eqn:computationmongecausality} is equal, so that the second and third term are equal and $\del{Y_t}_{t \in \sbr{0,1}}$ is a $\deln{\eta,\deln{\deln{\cH_t \otimes \cH_t}^{\pi_T}}_{t + \sbr{0,1}}}$-semimartingale. By the stochastic integral representation of elements in $\cM^2\deln{\pi_T, \deln{\cbrn{\emptyset, \cW^d} \otimes \cH^\eta_t}}$ from Proposition \ref{prop:degenMRPoncanonicalspace} and Theorem \ref{thm:equivctdscausality} $ii.$, we conclude $\pi_T \in \Pi_{bc}\del{\eta,\nu}$.
\end{proof}

\begin{proof}[Proof of Theorem \ref{thm:bicausaltptmapsbetweenSDEsdegen}]
We start with sufficiency. Suppose $T$ is a measurable map such that $\del{T_t}_{t \in \sbr{0,1}}$ is a $\deln{\eta,\deln{\cH^\eta_t}}$-semimartingale satisfying \eqref{eqn:bicausalmongemapeqn} and \eqref{eqn:Qstructureeqn}. By Lemma \ref{lemma:bicausalmongemapequivctd}, it is sufficient to show that $T_{\#}\eta = \nu$. As in the proof of Proposition \ref{prop:degenMRPoncanonicalspace}, up to enlarging the probability space $\deln{\cW^d, \cF_1, \eta}$, there exists some Brownian motion $B$ such that $\int_0^\cdot \sigma^\dag\del{s,X} dM^\eta_s = \int_0^\cdot \sigma^\dag\del{s,X} \sigma\del{s,X} dB_s$. Hence, Hence, \eqref{eqn:bicausalmongemapeqn} and \eqref{eqn:Qstructureeqn} implies $\eta$-a.s.  \begin{multline*}
     T_t = \tilde{z} + \int_0^t \overline{b}\del{s,T} ds + \int_0^t \overline{\sigma}\del{s,T} Q_s \sigma^\dag\del{s,X} \sigma\del{s,X} dB_s \\
     = \tilde{z} + \int_0^t \overline{b}\del{s,T} ds + \int_0^t \overline{\sigma}\del{s,T} Q_s dB_s.
     \end{multline*} Since $\int_0^\cdot Q_s dB_s$ is a Brownian motion, the law of $T$ is indeed $\nu$ by weak uniqueness. 

For necessity, if $T$ is a bicausal Monge map from $\eta$ to $\nu$, then by Lemma \ref{lemma:bicausalmongemapequivctd}, we have that $\del{T_t}_{t \in \sbr{0,1}}$ is a $\deln{\eta,\del{\cH^\eta_t}}$-semimartingale with martingale part $ \del{N_t}_{t \in \sbr{0,1}}$ where $N_t := T_t - \int_0^t \overline{b}\del{s,T} ds - \tilde{z}$ and $\eta$-a.s. \[
\sbr{T}_t = \int_0^t \overline{\sigma}\del{s,T} \overline{\sigma}^*\del{s,T} ds \text{ for } t \in \sbr{0,1}.
\] On the other hand, standard localization procedures and Proposition \ref{prop:degenMRPoncanonicalspace} yields a $\deln{\cH^\eta_t}_{t \in \sbr{0,1}}$-progressively measurable process $H : \sbr{0,1} \times \cW^d \rightarrow \bbR^{d \times d}$ such that for $t \in \sbr{0,1}$, $\eta$-a.s. \begin{equation}
    \label{eqn:MRPappliedtoT}
    N_t = \int_0^t H_s \sigma^{\dag}\del{s,X}dM^\eta_s 
\end{equation} with quadratic variation \[
\sbr{T}_t = \sbr{N}_t = \int_0^t \del{H_s \sigma^\dag\del{s,X} \sigma\del{s,X}} \del{H_s \sigma^\dag\del{s,X} \sigma\del{s,X}}^* ds \; \eta \text{-a.s..}
\] Hence, we obtain the $\eta$-a.s. equality \[
    \int_0^t \del{H_s \sigma^\dag\del{s,X} \sigma\del{s,X}} \del{H_s \sigma^\dag\del{s,X} \sigma\del{s,X}}^* ds = \int_0^t \overline{\sigma}\del{s,T} \overline{\sigma}^*\del{s,T} ds.
    \] By Lebesgue Differentiation Theorem, we obtain that $\eta$-a.s. on a set of full Lebesgue measure \[
    \deln{H_s \sigma^\dag\del{s,X} \sigma\del{s,X}} \deln{H_s \sigma^\dag\del{s,X} \sigma\del{s,X}}^* = \overline{\sigma}\del{s,T} \overline{\sigma}^*\del{s,T}.
    \]  By polar decomposition, this implies the existence of some $\del{\cH^\eta_{t}}_{t \in \sbr{0,1}}$-progressively measurable\footnote{To see progressively measurability of $Q$, one can show that polar decomposition can be done in a measurable way following \cite{azoff1974borel}.} $Q : \sbr{0,1} \times \cW^d \rightarrow \cO^d$ 
    such that $\eta$-a.s. on a set of full Lebesgue measure, we have \begin{multline*}
    H_s \sigma^\dag\del{s,X} \sigma\del{s,X} = \overline{\sigma}\del{s,T} Q_s 
    \implies H_s \sigma^\dag\del{s,X} 
    = \overline{\sigma}\del{s,T} Q_s \sigma^\dag\del{s,X} \\ 
    \text{ and } \overline{\sigma}\del{s,T} Q_s = \overline{\sigma}\del{s,T} Q_s \sigma^\dag\del{s,X} \sigma\del{s,X},    
    \end{multline*}
     from which we obtain the desired using \eqref{eqn:MRPappliedtoT}.
\end{proof}

\section{Applications}

\label{section:applications}

\subsection{Bicausal couplings induced by  Monge maps}
\label{subsection:bicausaltptplansinducedbybicausalmongemaps}

While Theorem \ref{thm:bicausaltptmapsbetweenSDEsdegen} gives us necessary and sufficient conditions for a measurable map $T$ to be a bicausal Monge transport, it does not tell us \emph{when} a bicausal transport plan is induced by a bicausal Monge map. We will now show that whether a bicausal coupling is induced by a Monge map depends on the existence of a strong solution to an associated SDE \eqref{eqn:extendedbrowniansde}.

In light of Corollary \ref{corollary:pushforwardofbicausalwiener}, we first study the case of bicausal couplings between Wiener measures:

\begin{prop}[Bicausal Monge couplings between Wiener measures]
    \label{prop:bicausaltptplansinducedbybicausalmongemapswienermeasures}
    Let $\pi \in \Pi_{bc}\deln{\mathbb{W}^d, \mathbb{W}^d}$ and $\rho : \sbr{0,1} \times \cW^d \times \cW^d \rightarrow \mathcal{C}^d$ be the $\deln{\deln{\mathcal{\cH}_t \otimes \mathcal{\cH}_t}^\pi}_{t+ \in \sbr{0,1}}$-progressively measurable correlation process of the marginal processes $X,Y$ under $\pi$: 
    $$ \rho(t,X,Y)=\frac{d[X,Y]}{dt},\qquad \pi-{\rm a.s.}. $$
    \begin{enumerate}
        \item If $\rho$ is $\mathcal{O}^d$-valued and there exists a strong, pathwise unique  solution to the SDE \begin{equation}
        \label{eqn:extendedbrowniansde}
            d\begin{pmatrix}
            Z \\
            \tilde{Z}
        \end{pmatrix}_t = \begin{pmatrix}
            1 \\ 
            \rho\deln{t,Z,\tilde{Z}} 
        \end{pmatrix} dB_t, \quad \begin{pmatrix}
            Z_0 \\
            \tilde{Z}_0
        \end{pmatrix} = 0.
    \end{equation}
         Then, $\pi \in \mathcal{T}_{bc}\deln{\mathbb{W}^d, \mathbb{W}^d}$.
        \item If $\pi$ is induced by a bicausal Monge map $T$, then $\rho$ is $\mathcal{O}^d$-valued and there exists a strong solution to \eqref{eqn:extendedbrowniansde}. 
    \end{enumerate}
\end{prop}
    \begin{remark}
        
  A technical and subtle point about the measurability of the correlation process is as follows: A priori, for $\pi \in \Pi\deln{\mathbb{W}^d, \mathbb{W}^d}$, the marginal processes $X$ and $Y$ will be $\deln{\deln{\mathcal{\cH}_t \otimes \mathcal{\cH}_t}^\pi}_{t+ \in \sbr{0,1}}$-Brownian motions and hence, $\rho$, constructed via the Lebesgue differentiation of the quadratic covariation $\sbr{X,Y}$, will be $\deln{\deln{\mathcal{\cH}_t \otimes \mathcal{\cH}_t}^\pi}_{t+ \in \sbr{0,1}}$-progressively measurable. Yet, in \eqref{eqn:extendedbrowniansde}, we consider $\rho$ as a component of a diffusion coefficient of an SDE, which one usually assumes to be $\del{\mathcal{F}_t \otimes \mathcal{F}_t}_{t \in \sbr{0,1}}$ 
    -progressively measurable. Fortunately, one can always find some $\tilde{\rho} : \sbr{0,1} \times E \times S \rightarrow \mathcal{C}^d$ $\pi$-indistinguishable from $\rho$ that is $\del{\mathcal{F}_t \otimes \mathcal{F}_t}_{t \in \sbr{0,1}}$-progressively measurable by using the fact that the completion of the canonical filtration of a Brownian motion is always right-continuous and \cite[Theorem $4.37.3$]{he1992semimartingale}. 
  \end{remark}
\begin{proof}
    $i.:$ Under $\pi$, $X$ and $Y$ are $\deln{\deln{\cH_t \otimes \cH_t}^\pi}_{t + \in \sbr{0,1}}$-Brownian motions such that $\sbr{X,Y} = \int_0^\cdot \rho\del{s,X,Y} ds$ by Theorem \ref{thm:equivctdscausality}. Using the fact that $\rho$ is $\cO^d$-valued, routine calculations show that $\sbrn{\int_0^\cdot \rho\del{s,X,Y} dX_s - Y}_\cdot \equiv 0$ and hence, $Y_t = \int_0^t \rho\del{s,X,Y} dX_s$ $\pi$-a.s. for all $t \in \sbr{0,1}$. In other words, $\del{X,Y}$ is a solution to \eqref{eqn:extendedbrowniansde}. By pathwise uniqueness, it follows that $Y$ must be adapted to $X$ and hence, $\pi \in \cT_{bc}\deln{\bbW^d, \bbW^d}$.

    $ii.:$ Since $\pi$ is induced by a bicausal Monge map, then Corollary \ref{corollary:bicausalMongemapbetweenWienermeasures} implies that $\rho$ must be $\mathcal{O}^d$-valued and as before, we again have $Y_t = \int_0^t \rho\del{s,X,Y} dX_s$ $\pi$-a.s.. Hence, \[
    T_t = \int_0^t \rho\del{s,X,T} dX_s, \quad \mathbb{W}^d\text{-a.s.}
    \] and $\del{T_t}_{t \in \sbr{0,1}}$ is adapted to $\deln{\cH^{\mathbb{W}^d}_t}_{t \in \sbr{0,1}}$ - the filtration of Brownian motion $X$, which implies the existence of a strong solution given by $\del{X_{\cdot}, T}$.
\end{proof}

\begin{remark}
    The preceding proposition almost gives a complete characterization of the set $\mathcal{T}_{bc}\deln{\mathbb{W}^d, \mathbb{W}^d}$ in terms of the existence of strong solutions to \eqref{eqn:extendedbrowniansde}. However, at present, the hypothesis of $ii.$ seems insufficient to guarantee a \emph{pathwise unique} strong solution to \eqref{eqn:extendedbrowniansde}. If we were to assume the hypothesis of $ii.$ and in addition, assume the weak uniqueness of \eqref{eqn:extendedbrowniansde}, 
 then the existence of a strong solution along with weak uniqueness implies the existence of strong, pathwise unique  solution by a result of A. S. Cherny \cite[Theorem $3.2$]{cherny2002uniqueness}. This holds for example \begin{enumerate}
     \item when $\rho$ is only a function of $X$ and $t$, so that weak uniqueness of the SDE is trivial;
     \item or when $\rho$ is only a function of $Y$ and $t$. Indeed, any solutions to the SDE \[
    d\tilde{Z}_t = \rho\deln{t,Z,\tilde{Z}} dB_t = \rho\deln{t,\tilde{Z}} dB_t  
    \] will always have law given by $\mathbb{W}^d$ since $\rho \in \cO^d$. Hence, the preceding SDE has weak uniqueness and by \cite[Theorem $3.1$]{cherny2002uniqueness}, the joint law $\deln{\tilde{Z},B}$ is unique, that is, SDE \eqref{eqn:extendedbrowniansde} has weak uniqueness.
 \end{enumerate} 

    Finally, we note that if $\gamma \in \Pi_{bc}\deln{\mathbb{W}^d, \mathbb{W}^d}$ is such that $\rho\del{X,Y,t} = \mathbbm{1}_{\cbr{Y_t > 0}} - \mathbbm{1}_{\cbr{Y_t \leq 0}}$, then \eqref{eqn:extendedbrowniansde} reduces to the well-known Tanaka's SDE from which we see that there cannot exist a strong solution. Hence, $\gamma$ is not induced by a bicausal Monge map. 
\end{remark}

For bicausal couplings between laws of SDEs with strong, pathwise unique  solutions, we have the following corollary: 

\begin{corollary}
    \label{corollary:bicausaltptplansinducedbybicausalmongemapssdes}
    Under the assumptions of Corollary \ref{corollary:strongsolnsdebicausalmongemap}, for any $\pi \in \Pi_{bc}\deln{\mu^{\sigma,b}_z,\mu^{\overline{\sigma},\overline{b}}_{\overline{z}}}$, there exists $\hat{\pi} \in \Pi_{bc}\deln{\mathbb{W}^d, \mathbb{W}^d}$ such that \[
    \pi = \del{F^{\sigma,b}, F^{\overline{\sigma}, \overline{b}}}_{\#} \hat{\pi}.
    \] Then, under the additional assumption that $\overline{\sigma}\del{s, \omega}$ is invertible for all $\del{s,\omega}$, $\pi \in \mathcal{T}_{bc}\deln{\mu^{\sigma,b}_z,\mu^{\overline{\sigma},\overline{b}}_{\overline{z}}}$ if and only $\hat{\pi} \in \mathcal{T}_{bc}\deln{\mathbb{W}^d, \mathbb{W}^d}$. 
\end{corollary}

\begin{proof}
    It is easy to verify that $\mathbb{W}^d$-a.s. $R^{\sigma,b}_z \circ F^{\sigma,b} = X_{\cdot}$ and by pathwise uniqueness, $\mu^{\sigma,b}_z$-a.s. $F^{\sigma,b} \circ R^{\sigma,b}_z = X.$. Similar relations hold for $R^{\overline{\sigma}, \overline{b}}_{\tilde{z}}$ and $F^{\overline{\sigma}, \overline{b}}$. In particular, $\deln{R^{\sigma,b}_z, R^{\overline{\sigma}, \overline{b}}_{\tilde{z}}}_{\#}\pi = \hat{\pi}$. Suppose now that $\hat{\pi} \in \mathcal{T}_{bc}\deln{\mathbb{W}^d, \mathbb{W}^d}$, i.e. there exist some measurable $T : \cW^d \rightarrow \cW^d$ such that $\del{X_{\cdot},T}_{\#}\mathbb{W}^d = \hat{\pi}$. Then, \begin{multline*}
         \pi = \del{F^{\sigma,b}, F^{\overline{\sigma}, \overline{b}}}_{\#} \hat{\pi} = \del{F^{\sigma,b}, F^{\overline{\sigma}, \overline{b}}}_{\#} \del{X_\cdot,T}_{\#}\mathbb{W}^d \\ 
         = \del{F^{\sigma,b}, F^{\overline{\sigma}, \overline{b}} \circ T}_{\#} \del{R^{\sigma,b}_z}_{\#} \mu^{\sigma,b}_x = \del{X_\cdot, F^{\overline{\sigma}, \overline{b}} \circ T \circ R^{\sigma,b}_z}_{\#} \eta.
    \end{multline*} Corollary \ref{corollary:strongsolnsdebicausalmongemap} gives that $F^{\overline{\sigma}, \overline{b}}_y \circ T \circ R^{\sigma,b}_x \in \mathcal{T}_{bc}\deln{\mu^{\sigma,b}_z,\mu^{\overline{\sigma},\overline{b}}_{\overline{z}}}$. Conversely, if $\pi \in \mathcal{T}_{bc}\deln{\mu^{\sigma,b}_z,\mu^{\overline{\sigma},\overline{b}}_{\overline{z}}}$, i.e. $\pi = \deln{X_\cdot, \hat{T}}_{\#}\mu^{\sigma,b}_z$ for some measurable $\hat{T}$, then $\hat{\pi} = \deln{X_\cdot, R^{\overline{\sigma}, \overline{b}}_{\tilde{z}} \circ \hat{T} \circ F^{\sigma,b}}_{\#}\bbW^d \in \cT_{bc}\deln{\bbW^d, \bbW^d}$.
\end{proof}

\subsection{Denseness of bicausal Monge maps among bicausal couplings}

We establish conditions under which couplings induced by bicausal Monge maps are dense in the space of bicausal couplings. We first show that $\mathcal{T}_{bc}\deln{\mathbb{W}^d, \mathbb{W}^d}$ is dense in $\Pi_{bc}\deln{\mathbb{W}^d, \mathbb{W}^d}$ for the topology of weak convergence of probability measures before moving on to the general case. We use the result of M. \'Emery  \cite[Proposition $2$]{emery2005certain}, which showed that the set of ``mutually adapted Brownian motions"  is dense in the set of ``jointly immersed Brownian motions", which is exactly equal to $\Pi_{bc}\deln{\mathbb{W}^d, \mathbb{W}^d}$ in light of Theorem \ref{thm:bicausaltptplanbetweenSDEs}.

\begin{prop}
    \label{prop:bicausalmongemapdenseinbicausaltptplanswiener}
    $\mathcal{T}_{bc}\deln{\mathbb{W}^d, \mathbb{W}^d}$ is dense in $\Pi_{bc}\deln{\mathbb{W}^d, \mathbb{W}^d}$ for the topology of weak convergence of probability measures.
\end{prop}

\begin{proof}
    By \cite[Proposition $2$]{emery2005certain}, given any $\pi \in \Pi_{bc}\deln{\mathbb{W}^d, \mathbb{W}^d}$, there exists Brownian motions $X^n$ and $Y^n$ on some abstract filtered probability space $\deln{\Omega, \mathcal{G}, \del{\mathcal{G}_t}_{t \in \sbr{0,1}}, \mathbb{P}}$ such that for each fixed $n$, $X^n$ and $Y^n$ generate the same filtration and the law of $\del{X^n,Y^n}$ weakly converges  to $\pi$ as $n$ tends to infinity. Since $X^n$ and $Y^n$ generate the same filtration, then there exists some $T^n : \cW^d \rightarrow \cW^d$ such that $T^n\del{X^n} = Y^n$ $\mathbb{P}$-a.s.. Since $Y^n$ is a Brownian motion, it is easy to see that $T^n_{
    \#}\text{Law}\del{Z^n} = T^n_{\#}\bbW^d = \bbW^d$. Lemma \ref{lemma:bicausalmongemapequivctd} gives that $T^n \in \mathcal{T}_{bc}\deln{\mathbb{W}^d, \mathbb{W}^d}$ and that $\text{Law}\del{Z^n,Y^n} = \del{X_\cdot, T^n}_{\#}{\mathbb{W}^d}$ weakly converges to $\pi$.
\end{proof}
\begin{remark}
A longer, but more explicit proof may be obtained along the lines of the proof of \cite[Proposition $2$]{emery2005certain} where $Y^n$ is constructed explicitly as a stochastic integral against $X^n$ of an $\cO^d$-valued process, in agreement with Corollary \ref{corollary:bicausalMongemapbetweenWienermeasures}.  
\end{remark}

In fact, since the Brownian motions constructed  in \cite[Proposition $2$]{emery2005certain} generate the same filtration, we can say more.

\begin{prop}
    \label{prop:bijectivebicausalmongemapdenseinbicausaltptplanswiener}
    $\mathcal{T}_{bcb}\deln{\mathbb{W}^d, \mathbb{W}^d}$ is dense in $\Pi_{bc}\deln{\mathbb{W}^d, \mathbb{W}^d}$ for the topology of weak convergence of probability measures.
\end{prop}

\begin{proof}
    In the proof of Proposition \ref{prop:bicausalmongemapdenseinbicausaltptplanswiener}, since $X^n$ and $Y^n$ generate the same filtration, the role of $X^n$ and $Y^n$ are completely symmetric, i.e. there exists some $\hat{T}^n : \cW^d \rightarrow \cW^d$ such that $\hat{T}^n\del{Y^n} = X^n$ $\mathbb{P}$-a.s. and $\hat{T}^n \in \mathcal{T}_{bc}\deln{\mathbb{W}^d, \mathbb{W}^d}$. It then holds that $X^n = \hat{T}^n \circ T^n\del{X^n}$ $\mathbb{P}$-a.s.. In other words, $\hat{T}^n \circ T^n = X_\cdot$ $\bbW^d$-a.s.. A similar statement holds for $T^n \circ \hat{T}^n$.
\end{proof}

We now establish conditions on $\sigma, \overline{\sigma},b$ and $\overline{b}$ so that $\mathcal{T}_{bc}\deln{\mu^{\sigma,b}_z, \mu^{\overline{\sigma}, \overline{b}}_{\tilde{z}}}$ is dense in $\Pi_{bc}\deln{\mu^{\sigma,b}_z, \mu^{\overline{\sigma}, \overline{b}}_{\tilde{z}}}$.
We associate to each $\alpha \in \mathscr{A}^{d,d'}$ 
the mapping $\Gamma_\alpha : \mathcal{W}^d \rightarrow \deln{\mathbb{R}^{d \times d'}}^{\sbr{0,1}}$ given by $\Gamma_\alpha\del{\omega} := \del{t \mapsto \alpha\del{t, \omega}}$. \begin{assumption}
    \label{assumption:ctsgamma}
    $\Gamma_\alpha$ maps into $C\deln{\sbr{0,1}, \mathbb{R}^{d \times d'}}$ and the mapping $\Gamma_\alpha : \cW^d \rightarrow C\deln{\sbr{0,1}, \bbR^{d \times d'}}$ is continuous with respect to the uniform topologies on both domain and codomain.
\end{assumption} 

The preceding assumption is satisfied when, for e.g., $\alpha\del{t, \omega}$ is continuous in $t$ and Lipschitz continuous in $\omega$, with Lipschitz constant independent of $t$. Any $\alpha$ given by $\alpha\del{t, \omega} := f\del{\omega\del{t}}$ for $f$ continuous would also satisfy the preceding assumptions.

\begin{prop}
    \label{prop:bicausalmongemapdenseinbicausaltptplans}

        Assume the SDEs \eqref{eqn:firstsde} and \eqref{eqn:2ndsde} have strong solutions satisfying  pathwise uniqueness, given by the Ito maps $F^{\sigma,b} : \deln{\mathcal{W}^d, \cF_1} \rightarrow \deln{\mathcal{W}^d, \cF_1}$ (resp. $F^{\overline{\sigma}, \overline{b}})$. If $\sigma, \overline{\sigma} \in \mathscr{A}^{d,d}$ and $b, \overline{b} \in \mathscr{A}^{d,1}$ satisfy Assumption \ref{assumption:ctsgamma} and $\sigma\del{s, \omega}$ is invertible for all $\del{s,\omega}$, $\mathcal{T}_{bc}\deln{\mu^{\sigma,b}_z, \mu^{\overline{\sigma}, \overline{b}}_{\tilde{z}}}$ is dense in $\Pi_{bc}\deln{\mu^{\sigma,b}_z, \mu^{\overline{\sigma}, \overline{b}}_{\tilde{z}}}$ for the topology of weak convergence of probability measures.
\end{prop}

\begin{proof}
   Let $\eta := \mu^{\sigma,b}_z$ and $\nu := \mu^{\overline{\sigma}, \overline{b}}_{\tilde{z}}$ and $\pi \in \Pi_{bc}\del{\eta,\nu}$. 
    Then, $\pi = \deln{F^{\sigma,b}, F^{\overline{\sigma}, \overline{b}}}_{\#}\hat{\pi}$ for some $\hat{\pi} \in \Pi_{bc}\deln{\mathbb{W}^d, \mathbb{W}^d}$ by Corollary \ref{corollary:pushforwardofbicausalwiener}. By Proposition \ref{prop:bicausalmongemapdenseinbicausaltptplanswiener}, there exists a sequence $\del{\hat{\pi}_n}_{n \geq 0} \subseteq \mathcal{T}_{bc}\deln{\mathbb{W}^d, \mathbb{W}^d}$ such that $\hat{\pi}_n = \deln{X_\cdot, \hat{T}_n}_{\#}\mathbb{W}^d$ for some $\hat{T}_n : \cW^d \rightarrow \cW^d$ and $\hat{\pi}_n \rightarrow \hat{\pi}$ as $n \rightarrow \infty$. Letting $E_i := \cW^d$ for $i = 1,\dots,4$ and $X_i : \bigtimes_{j = 1}^4 E_j \rightarrow E_i, X_i\del{x_1, \dots, x_4} := x_i$ denote the $i^{th}$-marginal process, we consider the probability measure on $\bigtimes_{i = 1}^4 E_i$ given by \begin{equation}
    \label{eqn:gammanconstruction}
    \gamma_n := \del{X_\cdot, F^{\overline{\sigma}, \overline{b}} \circ \hat{T}_n \circ R^{\sigma,b}_z, R^{\sigma,b}_z, \hat{T}_n \circ R^{\sigma,b}_z}_{\#}\eta \in \mathcal{P}\del{E_1 \times \dots \times E_4}.
    \end{equation}
     Then, the sequence $\del{\gamma_n}_{n \geq 0}$ is tight. Indeed, one checks that for each $n \geq 0$, $\pi_n := \del{X_1,X_2}_{\#}\gamma_n \in \mathcal{T}_{bc}\del{\eta,\nu} \subseteq  \Pi_{bc}\del{\eta, \nu}$ by Corollary \ref{corollary:strongsolnsdebicausalmongemap} and $\del{X_3,X_4}_{\#}\gamma_n = \hat{\pi}_n \in \Pi_{bc}\deln{\mathbb{W}^d, \mathbb{W}^d}$. Recall from Remark \ref{remark:causalcpls} that $\Pi_{bc}\del{\eta, \nu}$ and $\Pi_{bc}\deln{\mathbb{W}^d, \mathbb{W}^d}$ are compact (hence, tight by Prokohorov's theorem).
    We conclude that $\del{\gamma_n}_{n \geq 0}$ is a sequence of probability measures whose marginals lie in tight subsets of $\mathcal{P}\deln{\cW^d \times \cW^d}$ and hence, is tight itself (see \cite[Lemma $4.4$]{villani2008optimal}). By Prokhorov's theorem, we can assume, that up to a subsequence, $\gamma_n$ converges to some $\gamma$ as $n \rightarrow \infty$ such that $\del{X_1, X_2}_{\#}\gamma \in \Pi_{bc}\del{\eta, \nu}$ and $\del{X_3,X_4}_{\#}\gamma \in \Pi_{bc}\deln{\mathbb{W}^d, \mathbb{W}^d}$. In fact, since $\del{X_3,X_4}_{\#}\gamma_n = \hat{\pi}_n$, then $\del{X_3,X_4}_{\#}\gamma = \hat{\pi}$ by uniqueness of limits for weak convergence of probability measures. Since $\pi_n$ lies in $\mathcal{T}_{bc}\del{\eta,\nu}$ and converges to $\del{X_1,X_2}_{\#}\gamma$, it suffices now to show that \begin{equation}
        \label{eqn:desiredbicausaltptplan}
        \del{X_1,X_2}_{\#}\gamma = \deln{F^{\sigma,b}, F^{\overline{\sigma}, \overline{b}}}_{\#}\hat{\pi} = \pi. 
    \end{equation} 
    
     For each $n \geq 0$, we show $\gamma_n = \deln{\deln{F^{\sigma,b}, F^{\overline{\sigma}, \overline{b}}}, \text{Id}_{E_3 \times E_4}}_{\#}\hat{\pi}_n$, where $\text{Id}_{E_3 \times E_4}$ is the identity map on $E_3 \times E_4$. 
    Since $\del{X_3,X_4}_{\#}\gamma_n = \hat{\pi}_n$, it suffices to verify that on $\del{E_1 \times \dots \times E_4,\gamma_n}$, $X_1 = F^{\sigma,b}\del{X_3}$ and $X_2 = F^{\overline{\sigma},\overline{b}}\del{X_4}$ $\gamma_n$-a.s.. By definition of $\gamma_n$, it is easy to see that $X_2 = F^{\overline{\sigma}, \overline{b}}\del{X_4}$ $\gamma_n$-a.s.. Hence, we only focus on $\del{X_1,X_3}$. Consider $\kappa_n := \del{X_1,X_3}_{\#}\gamma_n = \del{X_\cdot, R^{\sigma,b}_z}_{\#}\eta$. 
    Since $X_1 = F^{\sigma,b} \deln{R^{\sigma,b}_z\deln{X_1}}$ $\eta$-a.s. by pathwise uniqueness, it follows that on $E_1 \times E_3$, $X_1 = F^{\sigma,b}_x\del{X_3}$ $\kappa_n$-a.s. and hence, $\gamma_n$-a.s. on $E_1 \times \dots \times E_4$.  

    To show that \eqref{eqn:desiredbicausaltptplan} holds, it suffices to prove that on $\del{E_1 \times \dots \times E_4,\gamma}$, $X_1 = F^{\sigma,b}\del{X_3}$ and $X_2 = F^{\overline{\sigma}, \overline{b}}\del{X_4}$ holds $\gamma$-a.s. since we know $\del{X_3,X_4}_{\#}\gamma = \hat{\pi}$. We only prove that $X_1 = F^{\sigma,b}_x\del{X_3}$ $\gamma$-a.s. since the analogous equality between $X_2$ and $X_4$ is proven in the exact same manner. Since $\gamma_n \rightarrow \gamma$ as $n \rightarrow \infty$, by Skorokhod's representation theorem (see for e.g. \cite[Theorem $2.7$]{ikeda2014stochastic}), there exists a probability space $\del{\Omega, \mathcal{G}, \mathbb{P}}$ with a sequence of $\bigtimes_{i = 1}^4 E_i$-valued random variables $S^n = \del{S^n_1, \dots, S^n_4}, n = 0, 1, \dots, \infty$ such that $S^n_{\#}\mathbb{P} = \gamma_n$ for $n = 0, 1, \dots$, $S^\infty_{\#}\mathbb{P} = \gamma$ and $S^n$ converges uniformly to $S^\infty$ $\mathbb{P}$-a.s..
    
    Consider for each $n \geq 0$, the filtration generated by $S^n_3$, denoted $\deln{\widehat{\cF}_t}_{t \in \sbr{0,1}}$. Since $S^n_{\#}\mathbb{P} = \gamma_n$, then $\del{S^n_1, S^n_3}_{\#}\mathbb{P} = \deln{F^{\sigma,b}_x, \text{Id}_{E_3}}_{\#}\bbW^d$. In other words, $S^n_3$ is a $\deln{\widehat{\cF}_t}_{t \in \sbr{0,1}}$-Brownian motion and $S^n_1$ is an $\deln{\widehat{\cF}_t}_{t \in \sbr{0,1}}$-adapted process such that $\mathbb{P}$-a.s. \begin{equation}
        \label{eqn:sdeforSn1andSn3}
        \del{S^n_1}_t = x + \int_0^t b\del{s,S^n_1}ds + \int_0^t \sigma\del{s,S^n_1} d\del{S^n_3}_s, \; \forall t \in \sbr{0,1}.
    \end{equation}
     We now want to take appropriate limits in the preceding equality as $n \rightarrow \infty$. Since $S^n$ uniformly converges to $S^\infty$, then $S^n_1$ converges uniformly to $S^\infty_1$. By Assumption \ref{assumption:ctsgamma} on $b$, one easily verifies that $\beta : \sbr{0,1} \times \mathcal{W}^d \rightarrow \mathbb{R}^d$ given by $\beta\del{t,\omega} := \int_0^t b\del{s,\omega} ds$ also lies in $\mathscr{A}^{d,1}$ and satisfies Assumption \ref{assumption:ctsgamma} again. Hence, $\Gamma_\beta$ is continuous from $\mathcal{W}^d$ to $\mathcal{W}^d$ and $\mathbb{P}$-a.s. $\Gamma_\beta\del{S^n_1} = \int_0^\cdot b\del{s,S^n_1} ds$ converges uniformly to $\Gamma_\beta\del{S^\infty_1} = \int_0^\cdot b\del{s,S^\infty_1} ds$. For the stochastic integral term, we will argue that \begin{equation}
        \label{eqn:cvgofstochasticintegrals}
    \int_0^\cdot \sigma\del{s,S^n_1} d\del{S^n_3}_s \rightarrow \int_0^\cdot \sigma\del{s,S^\infty_1} d\del{S^\infty_3}_s \text{ in probability }
    \end{equation}
     by verifying the conditions of \cite[Theorem $2.2$]{kurtz1991weak}. By Assumption \ref{assumption:ctsgamma} on $\sigma$, one easily verifies that $\sigma\del{\cdot, S^n_1}$ converges uniformly to $\sigma\del{\cdot, S^\infty_1}$ as $n \rightarrow \infty$ $\mathbb{P}$-a.s.. For each $n \geq 0$, we consider the pair $\deln{\sigma\del{\cdot,S^n_1}, S^n_3}$ which are $\deln{\widehat{\cF}_t}_{t \in \sbr{0,1}}$-adapted processes such that $\del{S^n_1, S^n_3}$ converges uniformly to $\del{S^\infty_1, S^\infty_3}$ $\mathbb{P}$-a.s. and hence, in probability. 
     Furthermore, since $\bbP$-a.s. $\sbr{S^n_3}_t = t$ is deterministic, \cite[$C2.2\del{i}$]{kurtz1991weak} is trivially satisfied. This verifies the assumptions of \cite[Theorem $2.2$]{kurtz1991weak} for $\delta = \infty$ and we conclude that \eqref{eqn:cvgofstochasticintegrals} holds. Taking limit in probability in \eqref{eqn:sdeforSn1andSn3}, we have the equality $\mathbb{P}$-a.s. \[
     \del{S^\infty_1}_t = x + \int_0^t b\del{s,S^\infty_1}ds + \int_0^t \sigma\del{s, S^\infty_1} d\del{S^\infty_3}_s \; \forall t \in \sbr{0,1}.
     \] The final step is to show that both $S^\infty_1$ and $S^\infty_3$ are semimartingales with respect to $\bbP$ and the filtration jointly generated by $S^\infty_1$ and $S^\infty_3$ so that we can conclude $S^\infty_1 = F^{\sigma,b}\del{S^\infty_3}$ $\mathbb{P}$-a.s. (and, hence $X_1 = F^{\sigma,b}\del{X_3}$ $\gamma$-a.s.) by pathwise uniqueness
     . Note that for each $n \geq 0$, $\del{X_1,X_3}_{\#}\gamma_n = \deln{F^{\sigma,b}, X_\cdot}_{\#}\mathbb{W}^d \in \mathcal{T}_{bc}\deln{\eta, \mathbb{W}^d} \subseteq \Pi_{bc}\deln{\eta, \mathbb{W}^d}$. Since $\Pi_{bc}\deln{\eta, \mathbb{W}^d}$ is compact and hence, closed, $\del{X_1,X_3}_{\#}\gamma$ is the limit of probability measures in $\Pi_{bc}\del{\eta, \mathbb{W}^d}$, and hence, also lies in $\Pi_{bc}\deln{\eta, \mathbb{W}^d}$. That is, $\del{S^\infty_1, S^\infty_3}_{\#}\mathbb{P} \in \Pi_{bc}\deln{\eta, \mathbb{W}^d}$. By Theorem \ref{thm:equivctdscausality}, it is straightforward to verify that bicausality implies the desired, that is, $S^\infty_1$ and $S^\infty_3$ are semimartingales with respect to $\bbP$ and the filtration jointly generated by $S^\infty_i, i = 1,2$.

\end{proof}

\begin{prop}
    \label{prop:bijectivebicausalmongemapdenseinbicausaltptplans}
     Under the assumptions of Proposition \ref{prop:bicausalmongemapdenseinbicausaltptplans}, if $\overline{\sigma}\del{s, \omega}$ is invertible for all $\del{s,\omega}$, then $\mathcal{T}_{bcb}\deln{\mu^{\sigma,b}_z, \mu^{\overline{\sigma}, \overline{b}}_{\tilde{z}}}$ is dense in $\Pi_{bc}\deln{\mu^{\sigma,b}_z, \mu^{\overline{\sigma}, \overline{b}}_{\tilde{z}}}$ for the topology of weak convergence of probability measures.
\end{prop}

\begin{proof}
    We repeat the proof of Proposition \ref{prop:bicausalmongemapdenseinbicausaltptplans} with the use of Proposition \ref{prop:bicausalmongemapdenseinbicausaltptplanswiener} replaced by Proposition \ref{prop:bijectivebicausalmongemapdenseinbicausaltptplanswiener}. We only need to check that if $\hat{T} \in \cT_{bcb}\deln{\bbW^d, \bbW^d}$, then $T := F^{\overline{\sigma},\overline{b}} \circ \hat{T} \circ R^{\sigma,b}_z \in \cT_{bcb}\deln{\mu^{\sigma,b}_z, \mu^{\overline{\sigma}, \overline{b}}_{\tilde{z}}}$. Let $\hat{S} : \cW^d \rightarrow \cW^d$ be such that $\hat{S} \in \cT_{bc}\deln{\bbW^d, \bbW^d}$ and $\bbW^d$-a.s. $\hat{T} \circ \hat{S} = \hat{S} \circ \hat{T} = X_\cdot$. It is straightforward to verify that $S := F^{\sigma,b} \circ \hat{S} \circ R^{\overline{\sigma}, \overline{b}}_{\tilde{z}}$ satisfies $T \circ S = X_\cdot$ $\mu^{\overline{\sigma}, \overline{b}}_{\tilde{z}}$-a.s. and $S \circ T = X_\cdot$ $\mu^{\sigma,b}_z$-a.s. as required.
\end{proof}

\subsection{Equivalence of bicausal Monge and Kantorovich optimal transport}
 Given some cost $c : \cW^d \times \cW^d \rightarrow \mathbb{R}_{\geq 0} \cup \cbr{+\infty}$, we define  \begin{enumerate}
    \item the bicausal Kantorovich problem: \begin{equation}
    \label{eqn:bckp}
    \tag{BKP}
    K\del{\eta,\nu} := \inf_{\pi \in \Pi_{bc}\del{\eta,\nu}} \int_{\cW^d \times \cW^d} c\deln{\omega, \omega'} d\pi\deln{\omega, \omega'} = \inf_{\pi \in \Pi_{bc}\del{\eta,\nu}} \mathbb{E}^\pi\del{c},
    \end{equation}
    \item the bicausal Monge problem: \begin{equation}
    \label{eqn:bcmp}
    \tag{BMP}
    M\del{\eta,\nu} := \inf_{\pi \in \mathcal{T}_{bc}\del{\eta,\nu}} \int_{\cW^d \times \cW^d} c\deln{\omega,\omega'} d\pi\deln{\omega,\omega'} = \inf_{\pi \in \mathcal{T}_{bc}\del{\eta,\nu}} \mathbb{E}^\pi\del{c}.
    \end{equation}
\end{enumerate}

Since $\mathcal{T}_{bc}\del{\eta,\nu} \subseteq \Pi_{bc}\del{\eta,\nu}$, we have $K\del{\eta,\nu} \leq M\del{\eta,\nu}.$ We can derive sufficient conditions for equality in light of Proposition \ref{prop:bicausalmongemapdenseinbicausaltptplans}:

\begin{corollary}
    \label{corollary:equalityofbckpandbcmp}
   Under the assumptions of Proposition \ref{prop:bicausalmongemapdenseinbicausaltptplans}, if $c : \cW^d \times \cW^d \rightarrow \mathbb{R}_{\geq 0}$ is continuous and $\envert{c\deln{\omega,\omega'}} \leq \envert{\phi\del{\omega}} + \envert{\psi\deln{\omega'}}$ for some $\phi \in L^1\del{\eta,\cF_1}$ and $\psi \in L^1\del{\nu,\cF_1}$, then $K\del{\eta,\nu} = M\del{\eta,\nu}$.
\end{corollary}

This proof of the preceding corollary is an adaptation of that of \cite[Lemma $5.14$]{nutz2022introduction} with minor changes, and is, hence, omitted.

\subsection{Examples}

We now derive some explicit solutions to \eqref{eqn:bckp}-\eqref{eqn:bcmp} in the case where one of the probability measures is the law of an SDE which has invertible diffusion coefficient and the other is the law of an SDE which has a deterministic drift coefficient and almost deterministic diffusion coefficient in a sense to be made precise later. In particular, we highlight how one cannot expect uniqueness of the optimal bicausal transport.

When $\eta$ and $\nu$ are weak solutions of SDEs,   for any $\pi \in \Pi_{bc}\del{\eta,\nu}$, marginal processes $\del{X_t}_{t \in \sbr{0,1}}$ and $\del{Y_t}_{t \in \sbr{0,1}}$ are $\deln{\pi, \deln{\deln{\cH \otimes \cH}^\pi}_{t+ \in \sbr{0,1}}}$-semimartingales by Theorem \ref{thm:equivctdscausality}. Inspired by \cite[Proposition $3.3$]{adaptedwassersteindistancesbackhoff}, we will consider costs that have separate dependence on the martingale and finite variation parts of $X$ and $Y$. Writing 
$$  X=M^\eta+ V^\eta, \quad Y=M^\nu+ V^\nu $$   the semimartingale decomposition of $X$ (resp. $Y$) into its martingale and finite variation parts  under   $\eta$ (resp. $\nu$) we consider a transport cost $c$  given by \begin{equation}
c\del{X, Y} := h\del{V^\eta - V^\nu} + g\deln{\Tr\del{\tbr{M^\eta - M^\nu}_1}},   \label{eq.separablecost} 
\end{equation}
for some measurable $h : \cW^d \rightarrow \mathbb{R}_{\geq 0} \cup \cbr{+\infty}$ and  monotone $g : \mathbb{R} \rightarrow \mathbb{R}_{\geq 0} \cup \cbr{+\infty}$.  
\begin{prop}
    \label{prop:solntobicausalmongeegs}
Let $\sigma \in \mathscr{A}^{d,d}$, $b \in \mathscr{A}^{d,1}$ be such that the SDE \begin{equation}
    \label{eqn:sdewithnondegendiff}
    dX_t = b\del{t,X} dt + \sigma\del{t,X} dB_t, \; X_0 = 0,    
    \end{equation}
has a unique weak solution $\eta := \mu^{\sigma,b}_0$, where  $\sigma\del{t,\omega}$ is invertible for all $\del{t,\omega}$. 
    
 Let $\overline{\sigma} \in \mathscr{A}^{d,d}$ be such that $\overline{\sigma}(t,\omega)=\kappa(t)u(t,\omega)$ with $u(t,\omega)\in {\cal O}^d$ and assume the SDE 
  \begin{equation}
    \label{eqn:sdewithdetqv}
dY_t = \overline{b}\del{t} dt + \overline{\sigma}\del{t,Y} dB_t, \; Y_0 = 0,        
    \end{equation}
 has a unique, strong solution with law $\nu := \mu^{\overline{\sigma}, \overline{b}}_0$. Denote the corresponding  It\^o  map   by $F^{\overline{\sigma}, \overline{b}}_0 : \cW^d \rightarrow \cW^d$. 
 
Then an optimal transport from $\eta$ to $\nu$ for a cost of the form \eqref{eq.separablecost}  is given by the bicausal Monge map 
    \[
    \cW^d \ni \omega \mapsto T\del{\omega} := F^{ \sqrt{\overline{\sigma} \overline{\sigma}^*},\overline{b}}_0\del{\int_0^\cdot Q_s \sigma^{-1}\del{s,X} dM^\eta\del{s} }
    \] 
    where $Q : \sbr{0,1} \times \cW^d \rightarrow \mathcal{O}^d$ is a $\deln{\cH^\eta_t}_{t \in \sbr{0,1}}$-progressively measurable process given by $Q_t\del{\omega} = Q\del{t,\omega} := V\del{t,\omega} U^*\del{t,\omega}$ for $U\del{t,\omega}, V\del{t,\omega} \in \cO^d$ appearing in the SVD of matrix $\sigma\del{s,\omega} \sqrt{\overline{\sigma} \overline{\sigma}^*\del{s}}$. In particular, \begin{multline*}
    K\del{\eta,\nu} = M\del{\eta,\nu} = \mathbb{E}^{\eta}\del{h\del{\int_0^\cdot b\del{s,X} ds - \int_0^\cdot \overline{b}\del{s} ds}} + \\
    \mathbb{E}^{\eta} \del{g\del{\int_0^1 \Tr\del{ \sigma\del{s,X} \sigma^*\del{s,X} + \overline{\sigma} \overline{\sigma}^*\del{s}} ds - 2 \int_0^1 \Tr\del{\sigma\del{s,X} \sqrt{\overline{\sigma} \overline{\sigma}^*\del{s}} Q\del{s,X}} ds}}.
    \end{multline*}

\end{prop}

\begin{remark}
    Note that the condition on $\overline{\sigma}$ implies that 
     $\overline{\sigma}(t,\omega)\overline{\sigma}^*(t,\omega) $ only depends on $t$ 
    but is more general than assuming $\overline{\sigma}$ depends on time only. For example, 
    \begin{multline*}
\overline{\sigma}\del{t,\omega} := \\
\begin{pmatrix}
    \cos\del{\gamma + \beta} \del{\frac{\alpha_1 + \alpha_2}{2}}^2 + \cos\del{\gamma - \beta} \del{\frac{\alpha_1 - \alpha_2}{2}}^2 & \sin\del{\gamma - \beta} \del{\frac{\alpha_1 - \alpha_2}{2}}^2 + \sin\del{\gamma + \beta} \del{\frac{\alpha_1 + \alpha_2}{2}}^2 \\
    \sin\del{\gamma - \beta} \del{\frac{\alpha_1 - \alpha_2}{2}}^2 - \sin\del{\gamma - \beta} \del{\frac{\alpha_1 + \alpha_2}{2}}^2 & \cos\del{\gamma + \beta} \del{\frac{\alpha_1 + \alpha_2}{2}}^2 + \cos\del{\gamma - \beta} \del{\frac{\alpha_1 - \alpha_2}{2}}^2
\end{pmatrix}    
\end{multline*} satisfies this condition. 
\end{remark}
\begin{remark}
    As the singular value decomposition is non-unique,  the solution to $\eqref{eqn:bcmp}$-$\eqref{eqn:bckp}$ is also in general non-unique. Furthermore, the assumptions of Proposition \ref{prop:solntobicausalmongeegs} assumes no regularity on the drift/diffusion coefficients and still concludes the equivalence of the bicausal Kantorovich and Monge problem, which is not covered by Corollary \ref{corollary:equalityofbckpandbcmp}. In the case $d = 1$, $\sigma, \overline{\sigma} \geq 0$, $h\del{\omega} := \norm{\omega}_{1 \text{-var}}$ and $g\del{x} := x$, the preceding result agrees with that of \cite[Proposition $3.3$]{adaptedwassersteindistancesbackhoff}. \
\end{remark}

\begin{proof}
    By rotation invariance of Brownian motion and weak uniqueness, we can without loss of generality replace $\overline{\sigma}\del{t,\omega}$ by the time-dependent deterministic function $\xi : \sbr{0,1} \rightarrow \bbR^{d \times d}, \xi\del{t} := \sqrt{\overline{\sigma}\overline{\sigma}^*\del{t}} = \sqrt{\kappa \kappa^*(t)}$.  
    
    Let $\pi \in \Pi_{bc}\del{\eta,\nu}$. By Theorem \ref{thm:bicausaltptplanbetweenSDEs}, there exists some stochastic basis $\del{\Omega, \mathcal{G}, \del{\mathcal{G}_t}_{t \in \sbr{0,1}}, \mathbb{P}}$ and $\deln{\cG_t}$-adapted processes $B, \tilde{B}, Z$ and $\tilde{Z}$ satisfying  Theorem \ref{thm:bicausaltptplanbetweenSDEs}. Then \begin{multline*}
    \mathbb{E}^\pi\del{c} = \mathbb{E}^{\mathbb{P}}\del{c\del{X - Y}}
    = \mathbb{E}^{\mathbb{P}}\deln{h\deln{\int_0^\cdot b\deln{s,Z} ds - \int_0^\cdot \overline{b}\del{s} ds}} + \\
    \mathbb{E}^{\mathbb{P}}\deln{g\deln{\Tr\del{\sbr{\int_0^\cdot \sigma\del{s,Z} dB_s - \int_0^\cdot \xi\del{s} d\tilde{B}_s}_1}}}.
    \end{multline*} Setting $\rho_t=d\sbrn{B, \tilde{B}}/dt \in \mathcal{C}^d$ as in Theorem \ref{thm:bicausaltptplanbetweenSDEs}, we have \begin{align*}
    &\Tr\del{\sbr{\int_0^\cdot \sigma\deln{s,Z} dB_s - \int_0^\cdot \xi\del{s} d\tilde{B}_s}_1} \\
    = &\int_0^1 \Tr\del{ \sigma\deln{s,Z} \sigma^*\deln{s,Z} + \overline{\sigma} \overline{\sigma}^*\del{s}} ds - 2 \int_0^1 \Tr\del{\sigma\del{s,Z} \xi\del{s} \rho_s} ds \\
    \geq &\int_0^1 \Tr\del{ \sigma\deln{s,Z} \sigma^*\deln{s,Z} + \overline{\sigma} \overline{\sigma}^*\del{s}} ds - 2 \int_0^1 \sup_{C \in \mathcal{C}^d} \Tr\del{\sigma\deln{s,Z} \xi\del{s} C} ds \\
    = &\int_0^1 \Tr\del{ \sigma\deln{s,Z} \sigma^*\deln{s,Z} + \overline{\sigma} \overline{\sigma}^*\del{s}} ds - 2 \int_0^1 \sup_{C \in \mathcal{O}^d} \Tr\del{\sigma\deln{s,Z} \xi\del{s} C} ds,
    \end{align*} where we have used \cite[Proposition $4.7$]{bion2019wasserstein} in the last equality. By the same proposition, \[
    \sup_{C \in \mathcal{O}^d} \Tr\del{\sigma\deln{s,Z} \xi\del{s} C} = \Tr\del{\sigma\deln{s,Z} \xi\del{s} Q\deln{s,Z}},
    \] where $Q$ is defined as in the statement of the proposition. Hence, by monotonicity of $g$ and temporarily suppressing the arguments of $\sigma$ and $b$, \begin{multline}
    \label{eqn:lowerbdbckpeg}
    \mathbb{E}^\pi\del{c} \geq \mathbb{E}^{\eta}\deln{h\deln{\int_0^\cdot b ds - \int_0^\cdot \overline{b}\del{s} ds}} \\
    + \mathbb{E}^{\eta} \deln{g\deln{\int_0^1 \Tr\del{ \sigma \sigma^* + \overline{\sigma} \overline{\sigma}^*\del{s}} ds - 2 \int_0^1 \Tr\del{\sigma \xi\del{s} Q\del{s,X}} ds}}.
    \end{multline} Since the RHS of \eqref{eqn:lowerbdbckpeg} no longer depends on $\pi$, taking infimum over $\pi \in \Pi_{bc}\del{\eta, \nu}$ gives that \eqref{eqn:bckp} is lower bounded by the RHS of \eqref{eqn:lowerbdbckpeg}. One easily checks that this lower bound is achieved by $T$ as defined in the statement of the proposition, which is a bicausal Monge map by Theorem \ref{thm:bicausaltptmapsbetweenSDEsdegen}.

\end{proof}

\paragraph{One dimensional case}
In one dimension, the bicausal Kantorovich problem between laws of SDEs $\mu^{\sigma,b}_0$ and $\mu^{\overline{\sigma}, \overline{b}}_0$ was studied in \cite{backhoffveraguas2024adaptedwassersteindistancelaws,robinson2024bicausal} for a variety of costs and regularity conditions on $\sigma,b, \overline{\sigma}$ and $\overline{b}$.
For example, if $b, \overline{b}: \bbR \rightarrow \bbR$ and $\sigma, \overline{\sigma} : \bbR \rightarrow \bbR_{\geq 0}$ are Lipschitz, then by \cite[Theorem $1.3$]{backhoffveraguas2024adaptedwassersteindistancelaws}, for any $p \geq 1$ and $c\del{\omega, \omega'} := \int_0^1 \envert{\omega_s - \omega'_s}^p ds$, an explicit solution to \begin{equation}
    \label{eqn:1dimKantorovichbicausal}
    \inf_{\pi \in \Pi_{bc}\deln
{\mu^{\sigma,b}_0, \mu^{\overline{\sigma},\overline{b}}_0}} \mathbb{E}^\pi\del{c}
\end{equation} is given by the so-called \emph{synchronous coupling} $\pi^* := \deln{F^{\sigma,b}, F^{\overline{\sigma}, \overline{b}}}_{\#}\bbW^d$. In other words, on some probability space $\deln{\Omega, \cF, \bbP}$ with Brownian motion $B = (B_t)_{t \in \sbr{0,1}}$, we construct solutions to \eqref{eqn:firstsde} and \eqref{eqn:2ndsde} driven by the \emph{common} noise $B$ as $Z = \del{Z_t}_{t \in \sbr{0,1}} := F^{\sigma,b}\deln{B}$ and $\tilde{Z} = \deln{\tilde{Z}_t}_{t \in \sbr{0,1}} = F^{\overline{\sigma},\overline{b}}(B)$. Then, \[
\bbE^{\bbP}c(Z, \tilde{Z})  = \bbE^{\bbP} \int_0^1 |Z_s - \tilde{Z}_s|^p ds = \inf_{\pi \in \Pi_{bc}\deln
{\mu^{\sigma,b}_0, \mu^{\overline{\sigma},\overline{b}}_0}} \mathbb{E}^\pi\del{c}.
\] If $\sigma\del{x} > 0$ for all $x \in \bbR$, then by Corollary \ref{corollary:strongsolnsdebicausalmongemap}, we see that the synchronous coupling is induced by a Monge map $F^{\overline{\sigma}, \overline{b}} \circ R^{\sigma,b}$, \[
\pi^* = \deln{\text{Id}, F^{\overline{\sigma}, \overline{b}} \circ R^{\sigma,b}}_{\#}\mu^{\sigma,b}_0
\] and $\eqref{eqn:bckp}=\eqref{eqn:bcmp}$ in agreement with Corollary \ref{corollary:equalityofbckpandbcmp}. The filtration used to define bicausal couplings in \cite{backhoffveraguas2024adaptedwassersteindistancelaws} is the canonical one $\del{\mathcal{F}_t}_{t \in \sbr{0,1}}$, instead of the right-continuous one used in this work, $\del{\mathcal{H}_t}_{t \in \sbr{0,1}}$. However, the Lipschitz assumption in this example implies that $\mu^{\sigma,b}_0$ and $\mu^{\overline{\sigma}, \overline{b}}_0$ are laws of strong Markov processes \cite[Theorem 8.7]{le2016brownian} and these two definitions of bicausal couplings coincide (see \cite[Remark $2.3$]{backhoffveraguas2024adaptedwassersteindistancelaws} or Remark \ref{remark:causalcpls} $ii.$).



\textbf{Acknowledgements}
The authors would like to thank Yifan Jiang, Akshay Hegde and Daniel Goodair for the numerous fruitful discussions and helpful comments.

\newpage

\appendix

\end{document}